\documentclass[a4paper, 11pt,reqno]{amsart}
\usepackage{amsthm,amssymb,amsmath,graphicx,psfrag,enumitem,mathrsfs}
\usepackage{mathtools}
\usepackage[foot]{amsaddr}
\usepackage[british]{babel}
\usepackage[babel]{microtype}
\usepackage[margin=1in]{geometry}

\usepackage[textsize=footnotesize]{todonotes}

\usepackage{algorithm}
\usepackage{algpseudocode}
\usepackage{tikz}
\usepackage{setspace}

\usetikzlibrary{backgrounds}

\usepackage[numbers]{natbib}

\makeatletter
\def\NAT@spacechar{~}
\makeatother

\usepackage{hyperref}
\usepackage[capitalise]{cleveref}
\hypersetup{colorlinks=true,
   citecolor=blue,
   filecolor=blue,
   linkcolor=blue,
   urlcolor=blue
}

\crefname{figure}{Figure}{Figures}
\Crefname{figure}{Figure}{Figures}

\allowdisplaybreaks

\newtheorem{definition}{Definition}[section]
\newtheorem{claim}{Claim}

\newtheorem{proposition}[definition]{Proposition}
\newtheorem{theorem}[definition]{Theorem}

\newtheorem{lemma}[definition]{Lemma}

\newtheorem{conjecture}[definition]{Conjecture}

\numberwithin{equation}{section}

\newcommand{\comment}[1]{}

\renewcommand{\epsilon}{\varepsilon}

\newcommand{\COMMENT}[1]{}
\renewcommand{\COMMENT}[1]{\footnote{\textcolor{blue!70!black}{#1}}} 

\title{A large number of $m$-coloured complete infinite subgraphs}

\author[A.~Gir\~{a}o]{Ant\'onio Gir\~{a}o}
\email{giraoa@bham.ac.uk}
\address{School of Mathematics, University of Birmingham, 
Edgbaston, Birmingham, B15 2TT, United Kingdom.}

\date{\today}
\begin{document}

\begin{abstract}
Given an edge colouring of a graph with a set of $m$ colours, we say that the graph is $m$-\textit{coloured} if each of the $m$ colours is used. 
For an $m$-colouring $\Delta$ of $\mathbb{N}^{(2)}$, the complete graph on $\mathbb{N}$, we denote by $\mathcal{F}_{\Delta}$ the set all values $\gamma$ for which there exists an infinite subset $X\subset \mathbb{N}$ such that $X^{(2)}$ is $\gamma$-coloured. Properties of this set were first studied by Erickson in $1994$. Here, we are interested in estimating the minimum size of $\mathcal{F}_{\Delta}$ over all $m$-colourings $\Delta$ of $\mathbb{N}^{(2)}$. Indeed, we shall prove the following result. There exists an absolute constant $\alpha > 0$ such that for any positive integer $m \neq \left\{ {n \choose 2}+1, {n \choose 2}+2: n\geq 2\right\}$, 
$|\mathcal{F}_{\Delta}| \geq (1+\alpha)\sqrt{2m}$, for any $m$-colouring $\Delta$ of $\mathbb{N}^{(2)}$. This proves a conjecture of Narayanan. We remark the result is tight up to the value of $\alpha$.
\end{abstract}

\maketitle

\section{Introduction}

Frank Ramsey~\cite{Ramsey} proved in the $30$'s that whenever the edges of a complete infinite countable graph are finitely coloured there exists a monochromatic infinite complete subgraph. Since then, numerous generalizations of this result have been proved and we shall refer the reader to the book of Graham, Rothschild and Spencer~\cite{GraRothSpen} for an overview of more recent results. 

In Ramsey theory one is concerned with finding large \textit{monochromatic} substructures in any finite colouring of a `rich' enough structure. Perhaps the most fundamental problem in the area is concerned with finding good estimates for the well known diagonal ramsey numbers. Unfortunately, the current best lower and upper bounds are still quite far apart (see~\cite{Thomason, Conlon, Erdos, Spencer}). 

In $1975$, Erd\H{o}s, Simonovits and S\'os~\cite{ErdosSosSimon}, started a new line of research, commonly known as \textit{Anti-Ramsey} theory. The problems in this area lie at the opposite end of Ramsey theory; here one is interested in finding large totally multicoloured (\textit{rainbow}) substructures. In particular, the authors studied the following problem. Given a graph $H$ and a positive integer $n\geq |H|$ what is the largest number $m(H,n)$ for which there exists a colouring of the edges of $K_n$ using at least $m$ distinct colours and which does not contain a rainbow copy of $H$. They found the correct order of growth of the function $m(H,n)$, for any fixed graph $H$. 

It is then natural to investigate what happens in between these two extremes and this is the line of enquiry we shall pursue in this note. We remark that Canonical Ramsey theory, which originates in a paper of Erd\H{o}s and Rado~\cite{CanonRamsey} helps understanding this dichotomy by studying colourings of $\mathbb{N}^{(2)}$ that use infinitely many colours. Erd\H{o}s and Rado theorem states that any colouring $\Delta$ of $\mathbb{N}^{(2)}$ must contain an infinite complete subgraph which is either monochromatic, rainbow or it induces one of two other possible very structured colourings.

We say $\Delta$ is an $m$-colouring of the edges of $\mathbb{N}^{(2)}$ if $\Delta$ uses \textit{exactly} $m$ colours, or in other words $\Delta\colon \mathbb{N}^{(2)} \twoheadrightarrow [m]$ is a surjective map. (As usual, we write $X^{(r)}$ for the set of $r$-subsets of a set $X$, and identify $X^{(2)}$ with the complete graph with vertex set $X$.) 

Let $\Delta\colon \mathbb{N}^{(2)} \twoheadrightarrow \mathcal{C}$ be an edge colouring of the complete graph on $\mathbb{N}$. For a subset $X\subset \mathbb{N}$, we define $\gamma_{\Delta}(X)$ or $\gamma(X)$ to be the size of the set $\Delta(X^{(2)})$. Moreover, we let \[\mathcal{F}_{\Delta}= \left \{ \gamma(X):  X\subset \mathbb{N} \text{ such that } X \text{ is infinite}\right \}. \] Our main goal in this paper is to study the size of $\mathcal{F}_{\Delta}$, given a finite colouring $\Delta$ of $\mathbb{N}^{(2)}$. This line of research has been pursued first by Erickson, later by Stacey and Weidl and recently by Kittipassorn and Narayanan. 

In $1996$, Erickson~\cite{Erickson} noted that for any $m$-colouring $\Delta$ ($m\geq 2$) of the complete graph on $\mathbb{N}$, the set $\{1,2, m\}$ is always contained in $\mathcal{F}_{\Delta}$. Note that it follows by Ramsey theorem that $1$ is always contained in $ \mathcal{F}_{\Delta}$, and by the assumption on $\Delta$, one trivially has that $m\in \mathcal{F}_{\Delta}$. In the same paper, Erickson conjectured that the only values which are guaranteed to be in $\mathcal{F}_{\Delta}$, for any $m$-colouring $\Delta$ are precisely $1,2$ and $m$.

\begin{conjecture}[Erickson]
Let $m>k>2$ be positive integers. Then there exists an $m$-colouring $\Delta\colon \mathbb{N}^{(2)} \twoheadrightarrow [m]$ such that $k \notin \mathcal{F}_{\Delta}$. 
\end{conjecture}

This conjecture was partially settled by Stacey and Weidl~\cite{StaceyWeidl}. They showed that for every $k$, there exists a constant $C_k$ such that for any $m\geq C_k$, there is an $m$-colouring $\Delta$ of $\mathbb{N}^{(2)}$ for which $k \notin \mathcal{F}_{\Delta}$. 
This shows it is a hopeless task to try to find particular values which must necessarily belong to $\mathcal{F}_{\Delta}$, for every finite colouring $\Delta$. Therefore, one may turn to the question of estimating how small the set $\mathcal{F}_{\Delta}$ can be, for an arbitrary $m$-colouring $\Delta$ of $\mathbb{N}^{(2)}$. To do so, let us define a function $\psi$ on the naturals where
\[ \psi(m)= \underset{\Delta\colon \mathbb{N}^{(2)}\twoheadrightarrow [m]}{\min}  |\mathcal{F}_{\Delta}|.\] 
This function $\psi$ was first studied by Narayanan~\cite{Bhargav} and he proved a general lower bound for $\psi(m)$ which is tight for specific values of $m$. 

\begin{theorem}[Narayanan]\label{Bhargav}
Let $n \geq 2$ be the largest natural number such that $m \geq  {n \choose 2}+1$. Then, $ \psi(m)\geq n$.
\end{theorem}

Note that this result implies $\psi(m)\geq \left \lfloor \sqrt{2m} \right \rfloor$. Moreover, it is not difficult to see that this theorem is tight for every integer $m$ of the form $ {n \choose 2} +1$, for some $n \geq 2$. This can been seen by taking a \textit{rainbow} complete graph on the first $n$ integers and colouring the remaining edges of $\mathbb{N}^{(2)}$ with a colour that has not appeared before. It is also easy to check that for this colouring of $\mathbb{N}^{(2)}$, $\mathcal{F}_{\Delta}=\left \{ {t \choose 2} +1: t\leq n \large \right \}$. 
Turning now to upper bounds of $\psi$, Narayanan proved that there is a subset $A\subset \mathbb{N}$ of `density' $1$, where $\psi(m)=o(m)$, for every $m\in A$ as $m\rightarrow  \infty$. Unfortunately, it still an open question whether $\psi(m)=o(m)$ as $m\rightarrow \infty$. 

Indeed, the behaviour of $\psi$ within intervals of the form $\left [{n \choose 2}+2, {{n+1} \choose 2} \right]$ is far from being understood. However, it follows by a slightly stronger version of Theorem~\ref{Bhargav} appearing in \cite{Bhargav}, that for any number $m\in \left [{n \choose 2}+2, {{n+1} \choose 2} \right]$, $\psi(m)\geq n+1$. Moreover, it is not hard to show that this is tight when $m={n \choose 2}+2$. 
 To see this, take as before, a rainbow colouring on the first $n$ positive integers and colour the edges between any two integers in $\left [n+1,\infty \right]$ with a new colour, finally colour every edge between an integer in $\{1,2,\ldots, n\}$ and the remaining integers with another distinct colour from those used before. It is clear this colouring uses ${ n\choose 2}+2$ colours and it is easy to check that in this case $\mathcal{F}_{\Delta}=\{1\}\cup \left \{{t \choose 2}+2 : t\leq n \right \}$. 

Observe that both colourings described above are quite unstable with respect to the size of $\mathcal{F}_{\Delta}$; if one tries to locally change any of these colourings by recolouring some edges with a new colour then the size of $\mathcal{F}$, increases by at least a factor greater than $1$. We will make use of this unstable behaviour of these \textit{extremal} colourings to prove our main theorem. 

The above observation supports Narayanan's conjecture that $\psi$ is far from being monotone and that it actually `suffers' huge jumps. In~\cite{Bhargav}, he stated the following beautiful conjecture. 
\begin{conjecture}
There is an absolute constant $c > 0$ such that $\psi\left ({n \choose 2} +3\right) > (1+c)n$, for all natural numbers $n \geq 2$.
\end{conjecture}
 
Our main aim in this paper is to prove this conjecture in the following slightly more general form. 
 
 \begin{theorem}~\label{main}
There exists a constant $\alpha>0$ such that $\psi(m)\geq (1+\alpha)\sqrt{2m}$, for every integer $m \notin \left \{ {n \choose 2} +1, {n \choose 2} +2 : n\geq 2 \right \}$.
\end{theorem}

Observe that the theorem implies that $\psi$ is not a monotone function since $\psi \left ({n \choose 2}+3 \right)$ is much bigger than $\psi \left ({n+1 \choose 2}\right )$. 

To see that our result is tight up to the value of $\alpha$, note that by making a small variation on the colourings described above, it is possible to construct $\left ({n \choose 2}+3\right )$-colourings $\Delta$ of $\mathbb{N}^{(2)}$ (for all $n\geq 4$) for which $\left |\mathcal{F}_{\Delta}\right|\leq \left (1+\frac{1}{2}\right)n=\frac{3}{2}\sqrt{2m}+O(1)$. To this end, take a rainbow colouring on the first $n$ integers and colour every edge between any two integers greater than $n$ with a colour not used before, say colour $1$. Finally, colour the remaining edges with an endpoint in $\left \{1,2,\ldots,\lfloor\frac{n}{2}\rfloor \right \}$ with a distinct colour, say colour $2$ and  colour the rest of edges (those incident with an integer in $\left \{\lceil\frac{n}{2}\rceil,\lceil\frac{n}{2}\rceil+1, \ldots, n\right \}$) with yet a new colour, distinct from all colours we have used before. We did not try to optimize the value of $\alpha$ in Theorem~\ref{main}.

\section{Notation, structure of the paper and an outline of the proof}
Our notation is standard. As usual, we write $[n]$ for $\{1,2,\ldots,n\}$, the set of the first $n$ natural numbers. We denote a surjective map $f$ from a set $X$ to another set $Y$ by $f:X\twoheadrightarrow Y$.
In this paper, by a \textit{colouring} of a graph we mean a colouring of the edges of the graph, unless stated otherwise. Moreover, as expressed above, by an $m$-\textit{colouring} of a graph, we mean a colouring which uses \textit{exactly} $m$ colours. Indeed, given a colouring $\Delta\colon \mathbb{N}^{(2)} \twoheadrightarrow \mathcal{C}$ of the complete graph on $\mathbb{N}$, a subset $X$ of $\mathbb{N}$ is $m$-coloured if $\Delta\left (X^{(2)}\right )$, the set of values attained by $\Delta$ on the edges with both endpoints in $X$ has size exactly $m$ colours. With a slight abuse of notation we shall denote by $\Delta(X)$ the set $\Delta\left (X^{(2)}\right)$. A colouring $\Delta$ is said to be a \textit{finite} colouring if $\mathcal{C}$ has finite order and an \textit{infinite} colouring, otherwise. As usual, given a colouring of $X^{(2)}$, we say $X$ is \textit{monochromatic} if all edges of $X^{(2)}$ have the same colour and $X$ is \textit{rainbow} if every edge has a \textit{distinct} colour. We shall denote by $(A,B)$ the complete bipartite graph between the sets $A$ and $B$ and denote by $\Delta\left (\left (A,B \right )\right ) $ the set of colours appearing on the edges of $(A,B)$.

We shall now describe roughly what is the strategy of the proof of our main theorem. The first step consists of reducing the problem to a finite setting. More precisely, we construct an $m$-colouring $\Delta'$ of a finite complete graph $G$ (together with one coloured $\textit{special}$ vertex $\{x\}$) from an $m$-colouring $\Delta$ of $\mathbb{N}^{(2)}$. Additionally, we show there is a $\textit{natural}$ correspondence between $\gamma$-coloured sets $X'\subset V(G)\cup \{x\}$ and $\gamma$-coloured infinite subsets $X\subset \mathbb{N}$. 
This reduction will allow us to apply an inductive argument on the number of colours $m$. Then, we show that a colouring of $K_n$ which is \textit{far} from being rainbow must necessarily induce many subgraphs spanning distinct number of colours. It turns out that the function $\rho: V(G) \rightarrow [m]$, where $\rho(v)$ is number of distinct colours appearing only on the edges incident with $v$, will be a very useful tool to carry out this analysis. Finally, after having gathered enough structural information on the colouring $\Delta'$, we perform a careful stability analysis to show $\Delta'$ must indeed span many colour sizes.

\section{Framework and a preliminary lemma}\label{sec:preliminary}

In this section, we show how we can pass from an $m$-colouring of $\mathbb{N}^{(2)}$ to a colouring of a finite complete graph using exactly $m$ colours and establish a correspondence between $\gamma$-coloured subgraphs of the complete graph and $\gamma$-coloured infinite subsets of $\mathbb{N}$. We shall also prove Lemma~\ref{lemma: finitecomplete} which will be useful later in the proof of our main theorem. Finally, at the end of the section, we show how our methods may be used to give alternative proofs of Theorem~\ref{Bhargav} and a theorem of Kittipassorn and Narayanan which appeared in~\cite{BhargavKittipassorn}. 

Suppose $\Delta\colon \mathbb{N}^{(2)} \twoheadrightarrow [m]$ is an $m$-colouring of $\mathbb{N}^{(2)}$. By Ramsey's Theorem, there exists an infinite set $X_1\subset \mathbb{N}$ all of whose edges are coloured with the same colour. Now, we claim we can find an infinite subset $X'\subset X_1$ together with a finite set $X\subset \mathbb{N}\setminus X'$ such that all $m$ colours appear in $\Delta((X\cup X')^{(2)})$ and more importantly, for every vertex $x\in X$, every edge between $x$ and $X'$ has the same colour. 
\begin{lemma}\label{lem: correspond}
Let $m\ge 2$ be a positive integer and let $\Delta\colon \mathbb{N}^{(2)} \twoheadrightarrow [m]$ be a $m$-colouring of $\mathbb{N}^{(2)}$. Then, we can find an infinite set $X'\subset \mathbb{N}$ and a finite set $X \subset (\mathbb{N}\setminus X')$ satisfying the following three properties.
\begin{itemize}
\item[(1)] $X'$ is a monochromatic infinite complete graph;
\item[(2)] For every vertex $x\in X$, every edge from $x$ to $X'$ has the same colour;
\item[(3)] $\Delta((X\cup X')^{(2)})=[m]$. 
\end{itemize}
\end{lemma}
\begin{proof}
To prove this lemma we shall consider the following iterative procedure. 
We start by setting $X'=X_1$ and $X=\emptyset$. If $\Delta((X'\cup X)^{(2)})$ spans all $m$ colours, we shall stop. Otherwise, there must exist some edge $e=(a,b)$ of colour $m' \notin \Delta((X'\cup X)^{(2)})$. First, suppose neither $a$ nor $b$ belongs to $X'\cup X$. In this case, we add both $a$ and $b$ to $X$ and we pass to an infinite subset $X''\subset X'$ in which every edge from $a$ to $X''$ has the same colour and similarly every edge from $b$ to $X''$ has the same colour and we set $X=X\cup \{a,b\}$ and $X'=X''$. Note that we have added the new colour $m'$ to $\Delta((X\cup X')^{(2)})$ without affecting the presence of any colour that appeared before.
 
Suppose now exactly one of the two vertices, say $a$, belongs to $X\cup X'$. If $a\in X$, we add $b$ to $X$ and, as before, we pass to an infinite subset $X''\subset X'$ where every edge from $b$ to $X''$ has the same colour. If $a\in X'$, then we remove $a$ from $X'$ and add both $a$ and $b$ to $X$ and as before, we pass to an infinite subset of $X''\subset X'$ in which every edge from $b$ to $X''$ has the same colour. Observe that by assumption on $X_1$, $a$ sends colour $1$ to every vertex in $X''$. Hence, by setting $X=X\cup \{a,b\}$ and $X'=X''$, we add the colour $m'$ to the set $\Delta((X \cup X')^{(2)})$. As the total number of colours is finite, this process must stop and we must have found our desired sets $X$ and $X'$. 
\end{proof}

Suppose we have found two sets $X$ and $X'$, as in Lemma~\ref{lem: correspond}. We shall now describe how to construct a colouring $\Delta'$ of the complete graph $G(X',X)=G$ on a vertex set of order $|X|+1$. 
Set $V(G)=X \cup \{x'\}$, where $x'$ will be called the \textit{special} vertex of $G$. The vertex $x'$ is obtained by identifying all vertices of $X'$, moreover, the colouring $\Delta'$ is obtained from the colouring $\Delta$ induced on $X\cup X'$. Finally, we colour the special vertex $x'$ with the \textit{unique} colour of the edges in $X'^{(2)}$. 

We shall call $\Delta'$ a \textit{special} colouring of the complete graph $G(X,X')$ if it consists of a colouring of the edges of $G$ together with a colouring of the special vertex $x'$.

Given a complete graph $G$ with a special vertex $x'$ and a special colouring $\Delta'$, we say with a slight abuse of notation that $\Delta'$ is an $m$-colouring of $G$ if the total number of colours used by $\Delta'$ is $m$ (including the colour of the special vertex). 
For a subset $S\subseteq V(G)$ with $x'\in S$, we let $\gamma(S)$ or $\gamma_{\Delta'}(S)=\left|\{\Delta'(S^{(2)})\cup \{\Delta'(x')\}\} \right|$. Finally, we define $\mathcal{G}'_{\Delta'}(G)=\{\gamma(S): S\subset V(G) \text{ and } x'\in S\}$. 

To conclude our reduction we establish the following correspondence between coloured complete subgraphs of $G$ and coloured infinite subgraphs of $\mathbb{N}^{(2)}$. 

\begin{proposition}\label{prop:correspond}
Let $m\geq 2$ be a positive integer and $\Delta\colon \mathbb{N}^{(2)} \twoheadrightarrow [m]$ an $m$-colouring of $\mathbb{N}^{(2)}$. Let $X,X'\subset \mathbb{N}$ be two subsets of $\mathbb{N}$, as in Lemma~\ref{lem: correspond}. Finally, let $G=G(X,X')$ be a coloured complete graph as constructed above with a special colouring $\Delta'$. Then, $\mathcal{G}'_{\Delta'}(G)\subseteq \mathcal{F}_{\Delta}$.
\end{proposition}

The proof of this proposition follows from the definition of $G(X,X')$. 

This reduction will allow us to restrict our attention to finite coloured complete graphs with a colouring of a \textit{special} vertex. 

We observe that given a special $m$-colouring $\Delta'$ of a finite complete graph $G=K_{n+1}$ with a special vertex $x'$, we may easily construct a colouring $\Delta$ of $\mathbb{N}^{(2)}$ where $\mathcal{F}_{\Delta}\subseteq \mathcal{G}'_{\Delta'}$. To see this, take a copy of the coloured complete graph $G\setminus\{x'\}$ and view it lying within the first $n$ natural numbers and colour every edge between any two naturals in $[n+1,\infty]$ with colour $\Delta'(x')$. For every $i\in [n]$, we colour every edge between $i$ and any natural in $[n+1,\infty]$ with the colour of the edge between the vertex corresponding to $i$ in $G\setminus\{x'\}$ and the special vertex $x'$. This observation together with Proposition~\ref{prop:correspond} imply that 
\begin{equation}\label{infintofin}
\psi(m)=\underset{G, m\text{-colouring }\Delta'}{\min} |\mathcal{G}'_{\Delta'}(G)|,
\end{equation}
where the minimum is taken over every finite complete graph $G$ with a special $m$-colouring $\Delta'$. It is clear that this minimum can actually be taken  over bounded size complete graphs, where the bound depends only on $m$. 

From now on, we shall only be considering complete graphs $G$ with a special $m$-colouring $\Delta'$. Our aim is to obtain a bound on the size of $\mathcal{G}'_{\Delta'}(G)$. To do so we shall need the following technical lemma. 

\begin{lemma}\label{lemma: finitecomplete}
Let $m,n\geq 2$ be positive integers. Let $G=K_{n+1}$ be a complete graph on $n$ vertices with a special vertex $x'$ and let $\Delta'$ be a special $m$-colouring of $G$. Then, we can find a sequence of pairwise disjoint subsets $A_1, A_2, \ldots, A_{\ell} \subset V(G)$, for some $\ell \in [n]$, satisfying the following properties.
\begin{enumerate}
\item[(1)] $A_1=\{x'\}$;
\item[(2)] For every $i\in [\ell]$, $|A_i|=1$ or $|A_i|=2$;
\item[(3)] $\left \{\gamma_{\Delta'}\left (\bigcup_{i=1}^{j} A_i\right ) \right \}_{j\in [\ell]}$ forms a \textit{strictly} increasing sequence and $\gamma_{\Delta'}\left (\bigcup_{j=1}^{\ell} A_j\right )=m$;
\item[(4)] Whenever $|A_j|=2$, the set of colours appearing on the edges of the bipartite graph $\left (A_{t}, \bigcup_{i=1}^{j-1} A_i \right)$ is contained in the set $\Delta'\left (\bigcup_{i=1}^{j-1} A_i\right )$, for every $t \geq j$;
\item[(5)] For every $i<j$ and for any two vertices $v \in A_i$, $w\in A_j$, $$\left |\Delta'\left (\left (v, \bigcup_{t=1}^{i-1} A_t \right)\right )\setminus \Delta'\left (\bigcup_{t=1}^{i-1} A_t \right) \right| \geq \left |\Delta'\left (\left (w, \bigcup_{t=1}^{i-1} A_t \right)\right )\setminus \Delta' \left (\bigcup_{t=1}^{i-1} A_t \right) \right|.$$ 
\end{enumerate} 
\end{lemma} 

\begin{proof}
Our proof of this theorem is similar in spirit to the proof of Lemma~\ref{lem: correspond}. We shall proceed iteratively, finding at step $i$ a new set $A_i$. The key idea is that we want to avoid as much as possible the choice of sets of order $2$, more precisely, if we can find a single element set $A_i$, for which the sequence $A_1,\ldots, A_i$ satisfies all properties $(1)-(5)$, then we shall choose such set. 

To begin this procedure, we set $A_1=\{x'\}$. Suppose we have already found the first $t$ sets $A_1,\ldots, A_t$, if $\Delta' \left (\bigcup_{i=1}^{t} A_i \right )$ contains all the $m$ colours, we stop. Otherwise, there must exist an edge $e=(a,b)$ whose colour does not appear in $\Delta'\left (\bigcup_{i=1}^{t} A_t \right)$. 

Suppose there is such an edge $e=(a,b)$ with exactly one endpoint, say $a$, belonging to the set $\bigcup_{i=1}^{t} A_i$. Then, we shall choose a vertex $x \in \left (V(G)\setminus \bigcup_{i=1}^{t} A_i \right )$, which maximizes the size of $\left ( \Delta'\left (\left (x, \bigcup_{i=1}^{t} A_i \right) \right )\setminus \Delta'\left (\bigcup_{i=1}^{t} A_i \right) \right)$. Set $A_{t+1}=\{x\}$. 

 Note that $A_{t+1}$ has order $1$ and the set $\Delta'\left (\bigcup_{i=1}^{t+1} A_i \right)$ strictly contains $\Delta'\left (\bigcup_{i=1}^{t} A_i \right)$. Hence, the properties $(1)-(3)$ hold for the sequence $A_1,A_2,\ldots, A_{t+1}$. Finally, we need to check the sequence $A_1,A_2,\ldots,A_{t+1}$ also satisfies properties $(4)$ and $(5)$. Property $(5)$ holds trivially since all sets $A_i$ were chosen so they maximize the number of extra colours that appear by adding $A_i$ to the previously constructed sets, $A_1,\ldots, A_{i-1}$. To see that property $(4)$ also holds, suppose $A_i$ has two elements, for some $i\leq t$. Then, the set of colours that appear on the edges of the bipartite graph $\left (x,\bigcup_{j=1}^{i-1} A_j \right)$ must be contained in set of colours that appear on the edges of $\left (\bigcup_{j=1}^{i-1} A_j \right)$, since otherwise we would have chosen $A_i$ to be a single element set. 

Suppose now that for every edge $e=(a,b)$ for which $\Delta(e) \notin \Delta'(\bigcup_{i=1}^{t} A_i)$, neither of the endpoints of $e$ belong to $\bigcup_{i=1}^{t} A_i$. Then, we set $A_{t+1}=\{a,b\}$. As before, properties $(1)-(3)$ follow easily. Moreover, $(4)$ holds since by assumption $\Delta'\left(A_{t+1}, \bigcup_{i=1}^{t} A_i \right) \subset \Delta'\left (\bigcup_{i=1}^{t} A_i \right)$. Finally, $(5)$ holds exactly by the same reason as before, namely because we always chose $A_t$ so to maximize the number of extra colours appearing on $\bigcup_{i=1}^{t} A_i$. Observe that since the number of colours is finite our sequence $A_1,\ldots, A_t$ must terminate after at most $m$ steps and at the end we must have $\gamma_{\Delta'}\left (\bigcup_{i=1}^{t} A_i \right)=m$.  
\end{proof}

\subsection{A consequence of Lemma~\ref{lemma: finitecomplete}} We shall now deduce the following result originally proved by Kittipassorn and Narayanan~\cite{BhargavKittipassorn} and which answered a question of Narayanan from~\cite{Bhargav}.

\begin{theorem}\label{thm: KittiNaraya}
Let $m\geq 2$ be a positive integer and let $\Delta\colon \mathbb{N}^{(2)} \twoheadrightarrow [m]$ be an $m$-colouring of the complete graph on $\mathbb{N}$ and suppose $n$ is a natural number such that $m > {n \choose 2}+1$. Then, $$\mathcal{F}_{\Delta}\cap \left({n \choose 2}+1,{n+1 \choose 2}+1\right]\neq \emptyset.$$
\end{theorem}

\begin{proof}[Sketch of the proof.]
 Let $ m > {n \choose 2} +1$ and let $\Delta\colon \mathbb{N}^{(2)}\twoheadrightarrow [m]$ be an $m$-colouring of $\mathbb{N}^{(2)}$. First, we construct a coloured complete graph $G$ with a special colouring $\Delta'$ obtained from the colouring $\Delta$, as we have seen at the beginning of this section. Secondly, let $A_1,\ldots, A_{\ell}$ be the ordered sequence obtained by Lemma~\ref{lemma: finitecomplete} from $G$. 
 Set $c(j)=\left |\Delta'\left (\bigcup_{i=1}^{j}A_i\right )\setminus \Delta'\left (\bigcup_{i=1}^{j-1} A_i \right) \right|$, for every $j\in [\ell]$. Observe $c(j)$ counts the number of new colours that appear when we add the set $A_j$ to the previously constructed sets. As we shall see, a simple inductive argument shows that for every $j\in[\ell]$, the following inequality holds. 
\begin{equation} \label{eq}
\sum_{i=1}^{j} c(i) \leq {j \choose 2} +1.
\end{equation}
Indeed, $c(1)=1= {1 \choose 2}+1$. We may assume by the inductive hypothesis that $\sum_{i=1}^{j} c(j) \leq { j \choose 2} +1$. Suppose first that $A_{j+1}$ has order two. From property $(3)$ and $(4)$ of Lemma~\ref{lemma: finitecomplete}, we have that $c(j+1)=c(j)+1\leq  {j \choose 2} +1 +1\leq  {j+1 \choose 2} +1$. Suppose now $A_{j+1}=\{w\}$. Let $2\leq t\leq j$ be the largest index for which $A_t$ has order two (assuming such $t\geq 2$ exists). Then, by property $(5)$ in Lemma~\ref{lemma: finitecomplete}, the set of colours appearing in the bipartite graph $\left (w,\bigcup_{i=1}^{t-1} A_i \right)$ is contained in $\Delta'\left (\bigcup_{i=1}^{t-1} A_i\right)$. This implies $c(j+1)\leq c(j)+ 2+ (j-t)\leq c(j)+j\leq {j+1 \choose 2}+1$. On the other hand, if no set has order $2$, then clearly $c(j+1)\leq c(j)+ j \leq {j+1 \choose 2}+1$, as we wanted to show. 

Let $\gamma_{i}=\gamma_{\Delta'}(\bigcup_{j=1}^{i} A_j)$, for every $i\in [\ell]$, and let $t\in [\ell]$ be the largest index for which $\gamma_t \leq {n+1 \choose 2}+1$. If $\gamma_{t}\in \left({n \choose 2}+1,{n+1 \choose 2}+1\right]$, we are done. So we may assume $\gamma_{t}\leq {n \choose 2}+1$. 

We need the following claim.
\begin{claim}\label{claim:gamma}
$\gamma_{t+1}\leq \gamma_{t}+n$.
\end{claim}

\begin{proof}
If $A_{t+1}$ has two elements then by property $(4)$, $\gamma_{t+1}=\gamma_{t}+1\leq  \gamma_{t}+n$. So we shall assume $A_{t+1}=\{x\}$.
As before, let $2\leq k\leq t$ be the largest index for which $A_k$ has order two (assuming such index $k$ exists, if not let $k=1$).

Suppose, for contradiction that $\gamma_{t+1}\geq \gamma_{t}+n+1$. 
Property $(5)$, allow us to construct an increasing sequence of $n$ indices $i_1<\ldots<i_{n}$ in $\{k+1,\ldots, t\}$ or ($n-1$ indices if $ \left | \Delta'\left (\left(x,A_k\right)\right) \right|=2$), where $c(i_q)\geq q$ (or $c(i_{q})\geq q+1$), for every $q\in [n]$ (or $q\in [n-1]$).
Therefore, $\gamma_{t}\geq 1+ \sum_{q=1}^{n} c(i_q) \geq 1+ {n+1 \choose 2 }$, which contradicts the assumption $\gamma_{t}\leq {n \choose 2}+1$ and the claim follows. 
\end{proof}
Finally, note that by the Claim~\ref{claim:gamma},  $\gamma_{t+1}\leq \gamma_{t}+n\leq { n+1 \choose 2} +1$, which is a contradiction on the maximality of $t$.   

\end{proof}
We remark that trivially Theorem~\ref{thm: KittiNaraya} implies Theorem~\ref{Bhargav}.

\section{Proof of the main theorem}

The aim of this section is to prove Theorem~\ref{main}. 
As we saw in Section~\ref{sec:preliminary}, in order to prove Theorem~\ref{main}, it suffices to show that for every $m\notin\{ {n \choose 2}+1,{n \choose 2}+2: n\geq 2\}$ and every special $m$-colouring $\Delta'$ of a complete graph $G$ with a special vertex $x'$,
$$|\mathcal{G}'_{\Delta'}(G)| \geq (1+\alpha)\sqrt{2m}.$$
 
To prove this, we shall need to introduce few more definitions.
We will assume from now on that the special vertex $x'$ has colour $1$. 

Given a vertex $v\in G\setminus \{x'\}$, we let $\rho_{G}(v)$ (or $\rho(v)$, whenever it is clear from the context what is the ground graph) to be the number of distinct colours which appear only on the edges incident with $v$, in other words, $\rho(v)$ counts the number of colours which cease to exist when the vertex $v$ is deleted. In particular, the colour $1$ never contributes to $\rho(v)$. Note that $\rho$ is not defined for the special vertex $x'$.

We denote by $G^{\geq 1}\subseteq G$ the complete subgraph induced on the set of vertices $\{v \in G': \rho(v)\geq 1\}\cup \{x'\}$ and we let $G'=G\setminus\{x'\}$. We shall call an edge $e\in G^{(2)}$ \textit{uniquely} coloured if $\Delta'(e)\neq \{1\}$ and no other edge shares the same colour as $e$. 
Moreover, we say that $\Delta'$ is a \textit{bad} (special) colouring of a complete graph $G=K_{n+1}$ (for some $n\geq 1$), if satisfies one of the following two types.
\begin{enumerate}
\item[$(i)$] $G$ is a rainbow complete graph and none of the edges in $G$ has colour $1$;
\item[$(ii)$] $G\setminus \{x'\}=G'$ forms a rainbow complete graph where every edge has a colour distinct from $1$. Moreover, every edge incident with $x'$ has either colour $1$ or colour $2$ and colour $2$ does not appear in $G'^{(2)}$. 
\end{enumerate}
Observe that if $m\notin \{ {n \choose 2}+1,{n \choose 2}+2: n\geq 2\}$ then no $m$-colouring can be \textit{bad}. 

Finally, let $\Delta'$ be a colouring of $K_{n+1}\cup\{v\}$ with special vertex $x'$ and suppose $\Delta'$ induces a bad colouring on $K_{n+1}$ (with special vertex $x'$) then we say $v$ is a \textit{copycat} vertex if there exists $z\in K_{n+1}\setminus\{x'\}$ such that the following holds.
\begin{enumerate}
    \item If $\Delta'$ induces a bad colouring of type $(i)$ on $K_{n+1}$, then for every vertex $w\in (K_{n+1}\setminus \{z,v\})$ (including the special vertex $x'$), $\Delta'((v,w))=\Delta'((z,w))$ and $\Delta'((v,z))=1$;
    \item If $\Delta'$ induces a bad colouring of type $(ii)$ on $K_{n+1}$, then for every vertex $w\in (K_{n+1}\setminus\{x',z\})$, $\Delta'((v,w))=\Delta'((z,w))$. Furthermore, $\Delta'((z,x')),\Delta'((z,v))\in \{1,2\}$ if colour $2$ already appears in $K_{n+1}$ or $\Delta'((z,x'))=\Delta((z,v))=1$, otherwise.
\end{enumerate}{}
 
We shall deduce Theorem~\ref{main} from the following stronger result. 
 
\begin{theorem}\label{thm:main2}
 There exists an absolute constant $\alpha'>0$ such that if $G=K_{n+1}$ is a complete graph with a special $m$-colouring ($m\geq 3)$ $\Delta'$ and $G=G^{\geq 1}$, then at least one of the following holds.
\begin{itemize}
\item $\Delta'$ is a \textit{bad} colouring;
\item $\left |\mathcal{G}'_{\Delta'}(G)\right|\geq (1+\alpha')\sqrt{2m}$.
\end{itemize}
\end{theorem} 
\begin{proof}[\textbf{Deducing Theorem}~\ref{main} \textbf{from Theorem}~\ref{thm:main2}]
Let $\Delta$ be an $m$-colouring of $\mathbb{N}^{(2)}$.  Construct a coloured complete graph $G$ with a special vertex $x'$ with a special $m$-colouring $\Delta'$, as seen in Section~\ref{sec:preliminary}. We shall now construct a sequence of decreasing subgraphs $G=G_0\supset G_1 \supset \ldots \supset G_k$ where $G_i$ is obtained from $G_{i-1}$ by deleting exactly one vertex $x_{i-1}\in (V(G_{i-1})\setminus\{x'\})$ with the property that $\rho_{G_{i-1}}(v)=0$. Since $G$ is finite this process must stop and $G_k=G_{k}^{\geq 1}$. Moreover, by construction notice that $ \gamma_{\Delta'}(G_k)=m$. Finally, observe that the colouring induced on $G_{k}$ can not bad. The conclusion of the theorem now follows from Proposition~\ref{prop:correspond} and Theorem~\ref{thm:main2}.
\end{proof}

The rest of this section will be devoted to proving Theorem~\ref{thm:main2}. 

For technical reasons, we shall need to prove our main theorem for small values of $m\geq 2$. Observe first that whenever $m\notin \left \{ {n \choose 2}+1 : n\geq 2 \right \}$, we have that $\psi(m)>\sqrt{2m}$. This follows trivially from Theorem~\ref{thm: KittiNaraya}. Now, let $\Delta'\colon G^{(2)}\twoheadrightarrow [m]$ be a special colouring of a finite graph $G$ with special vertex $x'$ where the colour of the vertex $x'$ is $1$. Suppose as well that $m={n\choose 2}+1$, for some $n\geq 2$ and $\Delta'$ is not a \textit{bad} colouring. Then, $\left |\mathcal{G}'_{\Delta'}(G)\right|\geq n+1 >\sqrt{2m}$. This follows from Lemma~\ref{lemma: finitecomplete}.
Indeed, let $A_1,A_2,\ldots, A_t$ be the sets obtained from Lemma~\ref{lemma: finitecomplete}, then if $t>n$, we are done. Suppose now $t=n$. Clearly, no $A_i$ can have order $2$ for $i\geq 3$ since in that case the contribution of new colours when adding a new set $A_j$ (for $j\geq i$) is strictly less than $j-1$ (by property $(4)$) which would force $t>n$, a contradiction. Therefore, either every $A_i$ has size $1$, in which case $\Delta'$ is a bad colouring of type $(i)$, or $A_2=2$ in which case every edge incident with $x'$ has colour $1$ (by property $(4)$) and $\Delta'$ is a bad colouring of type $(ii)$, a contradiction.  
These couple of observations imply the following proposition. 

\begin{proposition}\label{prop:basecase}
There exists a constant $\alpha>0$ such that for any $2\leq m\leq 200$ the following holds. Let $\Delta'\colon G^{(2)} \twoheadrightarrow [m]$ be a special colouring of a complete graph with special vertex $x'$ of colour $1$ and suppose $\Delta'$ is not a bad colouring. Then,
$\left |\mathcal{G}'_{\Delta'}(G)\right| \geq (1+\alpha)\sqrt{2m}$.
\end{proposition}
We shall use this fact as the base case for the induction argument employed in the proof of our main theorem. 

We move on by proving three lemmas which will be useful later in the proof of our main theorem. They assert that Theorem~\ref{thm:main2} holds whenever the special colouring of $G$ is close to being a bad colouring but it is not bad.  

\begin{lemma}\label{lem:uniquelycolour}
Let $m\geq n\geq 15$ be positive integers and let $\rho \geq 6$. Let $\Delta': G^{(2)} \twoheadrightarrow [m]$ be a special $m$-colouring of a complete graph $G=K_{n+1}$ with special vertex $x'$ of colour $1$. Suppose $\Delta'$ is not a bad colouring and every edge in $G'^{(2)}=(G\setminus\{x'\})^{(2)}$ is uniquely coloured. Moreover, suppose $\rho(v)=\rho$, for every $v\in V(G')$. Then, $$|\mathcal{G}'_{\Delta'}(G)|\geq \frac{5}{4}n. $$
\end{lemma}

\begin{proof}
Observe that $\rho\geq n-1$, since every edge in $G'^{(2)}$ is uniquely coloured. If $\rho=n$, then the colouring $\Delta'$ would have to be a bad colouring of type $(i)$, which is not possible. Therefore, we must have $\rho=n-1$. Let $A=\{a_1,a_2,\ldots, a_k\}$ be the set of distinct colours (including perhaps colour $1$) appearing on the edges incident with $x'$. As $\Delta'$ is not a bad colouring, $k\geq 2$, whenever $1\notin A$ or $k\geq 3$, otherwise.  

Let us partition $V(G')$ into $k$ sets, namely $V_1, V_2,\ldots, V_k$, such that for every $v\in V_i$, the colour of the edge $(v,x')$ is $a_i$. 
Note that since $\rho=n-1$, each $V_i$ must have size at least $2$, whenever $a_i\neq 1$. Since $k\geq 2$, we may assume the size of $V_k$ is smaller or equal than $\frac{n}{2}$. Let $x'=v_1<v_2<\ldots<v_{n+1}$ be an ordering of the vertices of $G$ with the property that every vertex of $V_i$ comes before every vertex of $V_j$, for every $i<j\in [k]$. Furthermore, if we denote by $T_i$ the set of the first $i$ vertices in this ordering, we clearly have $\gamma_{\Delta'}(T_i)<\gamma_{\Delta'}(T_j)$, for every $i<j$. 

Consider the set $S\subseteq [n+1]$ consisting of every $s \in [n+1]$ such that $v_s\notin \left (V_{k}\cup\{x'\} \right)$ and if $v_s \in V_{\ell}$ then $v_s$ is \textbf{not} the first vertex of $V_{\ell}$ in the above ordering. 

Let $T^{*}_s=T_s \setminus\{v_s\}\cup\{v_{n+1}\}$, for every $s\in S$. As before, it is easy to see that $\gamma_{\Delta'}(T^{*}_i)<\gamma_{\Delta'}(T^{*}_j)$, for every $i<j\in S$. Finally, note that $\gamma_{\Delta'}(T^{*}_s)\neq \gamma_{\Delta'}(T_p)$, for every $s\in S$ and $p \in [n]$. This holds since by construction $\gamma_{\Delta'}(T^{*}_s)= \gamma_{\Delta'}(T_s)+1$, for every $s\in S$, and the difference between $\gamma_{\Delta'}(T_s)$ and $\gamma_{\Delta'}(T_{s+1})$ is greater than $1$, for every $s\geq 2$. Hence, $$|\mathcal{G}'_{\Delta'}(G)| \geq n+|S|\geq n+\left(\frac{n-|V_k|}{2}\right)\geq \frac{5}{4}n.$$.
\end{proof}

The next two lemmas are slightly technical and its proofs follow a careful analysis of different possible cases.
\begin{lemma}\label{lem:type 1}
Let $m\geq n\geq 10$ be positive integers. Let $\Delta': G^{(2)} \twoheadrightarrow [m]$ be a special $m$-colouring of a complete graph $G=K_{n}\cup\{x'\}\cup\{v\}$ with a special vertex $x'$ of colour $1$. Moreover, suppose $\Delta'$ induces a bad colouring of type $(i)$ on $G\setminus\{v\}$ (with special vertex $x'$). Then, one of the following holds. 
\begin{enumerate}
    \item $ \left |\mathcal{G}'_{\Delta'}(G) \right|\geq \frac{5}{4}n$;
    \item $\Delta'$ induces a bad colouring of type $(i)$ on the entire $G$;
    \item The edge $(v,x')$ has colour $1$ and for every vertex $x \in G\setminus\{x',v\}$, the edge $(v,x)$ has either colour $1$ or colour $\Delta'((x',x))$.
    \item $v$ is a \textit{copycat} vertex. 
    \end{enumerate}
\end{lemma}

\begin{proof}
For convenience of notation, we shall denote by $H$ the graph $G\setminus\{v\}$, note that $|H|=n+1$. For the proof, it will be easier to assume $(2),(3)$ and $(4)$ do not hold and we shall deduce $(1)$ must then be satisfied.

Note that by considering any subset $W\subset H$ (containing $x'$) of size $\ell$, we deduce that
\begin{equation}\label{eq:obvious}
    \left \{{\ell \choose 2}+1: 1 \leq \ell\leq n+1\right \}\subset \mathcal{G}'_{\Delta'}(G).
\end{equation}
Let us now write $V(H)=\{x_0=x',x_1,\dots, x_{n}\}$. Since $(3)$ does not hold, there must exist a vertex $u\in H$ such that the colour of $(v,u)$ is not $1$ or $\Delta'((x',u))$. Say the edge $(v,u)$ has colour $i$.

Note that since $H$ is a rainbow clique there is a vertex, say $x_{n}\in (H\setminus \{x',u\})$, such that $H\setminus\{x_n\}$ does not span colour $i$. 

First, suppose that the colouring of $(H\setminus\{x_{n}\})\cup \{v\}$ is \textbf{not} a bad colouring of type $(i)$. This implies that there exist either an edge $(v,w)$ of colour $1$, for some $w\in H\setminus\{x_{n}\}$, or an edge $e=(v,w)$ of the same colour as some other edge $(a,b)\in (H\setminus \{x_{n}\})\cup \{v\}$. In the former case set $X_0=\{x',v,w,u\}$ and set $X_0=\{x',v,w,a,b,u\}$ for the latter. Note that possibly the set $X_0$ contains repeated vertices. In any case, the idea is that the set $X_0$ does not induce a bad colouring of type $(i)$ but also it contains at least one more colour than $X_0\setminus\{v\}$, namely colour $i$. Consider now the sets
$X_{\ell}= W_{\ell} \cup X_0$, for every $\ell\in \{1,\ldots,n-5\}$, where $W_{\ell}\subset H\setminus (X_0\cup \{x_n\})$ is any set of size $\ell$. It is easy to see that 
$ {|X_{\ell}|-1 \choose 2 } +2 \leq \left |\Delta'(X_{\ell}) \right| \leq  {\left|X_{\ell}\right| \choose 2}$. 

The upper bound holds since $X_{\ell}$ is not rainbow and the lower bound holds because $X_{\ell} \setminus\{v\}$ forms a rainbow clique and the edge $(v,u)$ has colour $i$ which does not appear in $X_{\ell}\setminus\{v\}$. Recalling ~(\ref{eq:obvious}), we have that $\mathcal{G}'_{\Delta'}(G)$ has order at least $2n-4\geq \frac{5n}{4}$, which is what we wanted show. 

Suppose now that the colouring of $(H\setminus\{x_{n}\})\cup \{v\}$ is a bad colouring of type $(i)$. Now, either the set of colours of the bipartite graph $(v, H\setminus\{x_n\})$ is the same as the set of colours in the bipartite graph $(x_n, H\setminus\{x_n\}) $ or not. We split our analysis in each of these two cases.\\

\textbf{Case $1$. $\Delta'((v, H\setminus\{x_n\}))\neq \Delta'((x_n, H\setminus\{x_n\}$)).}\\

By hypothesis, there is $w\in H\setminus\{x_n\}$ such that $\Delta'((v,w))$ has colour $j\neq 1$ and this colour does not appear in $\Delta'((x_n, H\setminus\{x_n\}))$. Observe the only other edge which can have colour $j$ is $(v,x_n)$. If $\Delta'((v,x_n))=j$ then set  $X_0=\{x',w,v,x_n\}$, and consider the sets
$X_{\ell}= W_{\ell} \cup X_0$, for every $\ell\in \{1,\ldots,n-3\}$, where $W_{\ell}\subset H\setminus X_0$ is any set of size $\ell$. It is easy to see that 
$ {|X_{\ell}|-1 \choose 2 } +2 \leq \left |\Delta'(X_{\ell}) \right| \leq  {\left|X_{\ell}\right| \choose 2}$, implying that $\mathcal{G}'_{\Delta'}(G)$ has order at least $2n-2\geq \frac{5n}{4}$, which is what we wanted show. If $\Delta'((v,x_n))\neq j$ then colour $j$ appears uniquely in the edge $(v,w)$. Since $(2)$ does not hold either $\Delta'((v,x_n)=1$ or there exist two distinct edges 
$e_1=(a,b),e_2=(a',b')$ in $H\cup\{v\}$ such that $\Delta'(e_1)=\Delta'(e_2)$. In the former case, set  $X_0=\{x',w,v,x_n\}$, and consider the sets
$X_{\ell}= W_{\ell} \cup X_0$, for every $\ell\in \{1,\ldots,n-3\}$, where $W_{\ell}\subset H\setminus X_0$ is any set of size $\ell$. It is easy to see that 
$ {|X_{\ell}|-1 \choose 2 } +2 \leq \left |\Delta'(X_{\ell}) \right| \leq  {\left|X_{\ell}\right| \choose 2}$, implying that $\mathcal{G}'_{\Delta'}(G)$ has order at least $2n-2\geq \frac{5n}{4}$, which is what we wanted show. In the latter case, set  $X_0=\{x',w,v,a,b,a',b'\}$ ($X_0$ may have repeated vertices), and consider the sets
$X_{\ell}= W_{\ell} \cup X_0$, for every $\ell\in \{1,\ldots,n-6\}$, where $W_{\ell}\subset H\setminus X_0$ is any set of size $\ell$. It is easy to see that 
$ {|X_{\ell}|-1 \choose 2 } +2 \leq \left |\Delta'(X_{\ell})\right| \leq  {\left|X_{\ell}\right| \choose 2}$, implying that $\mathcal{G}'_{\Delta'}(G)$ has order at least $2n-5\geq \frac{5n}{4}$, which is what we wanted show. This finishes the analysis of \textbf{Case $1$}. \\

\textbf{Case $2$. $\Delta'((v, H\setminus\{x_n\}))=\Delta'((x_n, H\setminus\{x_n\}$)).}\\

Let us assume there is a vertex $y\in H\setminus\{x_n\}$ such that $\Delta'((v,y))\neq \Delta'((x_n,y))$. As argued before there is a vertex $w\in H\setminus\{x',x_n\}$ with the property $H\setminus\{w\}$ does not span colour $\Delta'((v,y))$. Also, by hypothesis and the fact $n\geq 10$, there exist $a,b \in H\setminus\{x_n,w\}$ satisfying $\Delta'((v,a))=\Delta'((x_n,b))$ (possibly $a=b$). In which case, we set $X_0=\{x', v,x_n,y,a,b\}$ and as before we consider the sets
$X_{\ell}= W_{\ell} \cup X_0$, for every $\ell\in \{1,\ldots,n-5\}$, where $W_{\ell}\subset H\setminus \{X_0\cup \{w\}\}$ is any set of size $\ell$. It is easy to see that 
$ {|X_{\ell}|-1 \choose 2 } +2 \leq \left |\Delta'(X_{\ell}) \right| \leq  {\left|X_{\ell}\right| \choose 2}$, implying that $\mathcal{G}'_{\Delta'}(G)$ has order at least $2n-5\geq \frac{5n}{4}$, which is what we wanted show.
We may then assume for every $y\in H\setminus\{x_n\}$, the colour of $(v,y)$ is the same as the colour of $(x_n,y)$. Since $(4)$ does not hold, we may assume $\Delta'((v,x_n)=i\neq 1$. Moreover, there is a vertex $w\in H\setminus\{x_n\}$ with the property that $(H\setminus\{w\})$ does not span colour $i$. In this case, we set $X_0=\{x',v,x_n\}$ and as before we consider the sets
$X_{\ell}= W_{\ell} \cup X_0$, for every $\ell\in \{1,\ldots,n-3\}$, where $W_{\ell}\subset H\setminus X_0$ is any set of size $\ell$. It is easy to see that 
$ {|X_{\ell}|-1 \choose 2 } +2 \leq \left |\Delta'(X_{\ell}) \right| \leq  {\left|X_{\ell}\right| \choose 2}$, implying that $\mathcal{G}'_{\Delta'}(G)$ has order at least $2n-4\geq \frac{5n}{4}$, which is what we wanted show.
This concludes the proof of Lemma~\ref{lem:type 1}. 
\end{proof}

We shall need to prove an analogous lemma but now when the entire graph except one vertex induces a bad colouring of type $(ii)$. 

\begin{lemma}\label{lem:type 2}
Let $m\geq n\geq 15$ be positive integers. Let $\Delta': G^{(2)} \twoheadrightarrow [m]$ be a special $m$-colouring of a complete graph $G=K_{n}\cup\{x'\}\cup\{v\}$ with a special vertex $x'$ of colour $1$. Moreover, suppose $\Delta'$ induces a bad colouring of type $(ii)$ on $G\setminus\{v\}$ (with special vertex $x'$). Then, one of the following holds. 
\begin{enumerate}
    \item $ \left |\mathcal{G}'_{\Delta'}(G) \right|\geq \frac{5}{4}n$;
    \item $\Delta'$ induces a bad colouring of type $(ii)$ on the entire $G$;
    \item For every $w\in G$ (including $x'$), the edge $(v,w)$  has either colour $1$ or $2$; moreover, if colour $2$ does not already appear in $\Delta'(G\setminus \{v\})$, then all the edges $(v,w)$ have colour $1$;
    \item $v$ is a \textit{copycat} vertex.
    \end{enumerate}
\end{lemma}

\begin{proof}
The proof of this lemma follows a very similar strategy as the strategy employed in the proof of Lemma~\ref{lem:type 1}. As before, we shall denote by $H$ the graph $G\setminus\{v\}$ and observe that $|H|=n+1$. We shall assume hereafter that $(2),(3)$ and $(4)$ do not hold and our aim is to prove $(1)$ must happen. 
Let us write $V(H)=\{x_0=x',x_1,\dots, x_{n}\}$ and let $k\in [n]$ be a positive integer such that every edge $(x',x_i)$ has colour $1$ if $1\leq i\leq k$, and it has colour $2$ otherwise.

Observe that
\begin{equation}\label{eq:obvious1}
    \left \{{\ell \choose 2}+1: 1 \leq \ell\leq k\right \} \cup  \left \{{\ell \choose 2}+2: 2 \leq \ell\leq n+1\right \}         \subset \mathcal{G}'_{\Delta'}(G).
\end{equation}
Indeed, any subset $W \subset \{x',x_1,\ldots,x_k\}$ (containing $x'$) of size $\ell+1$ spans exactly ${\ell \choose 2}+1$ colours. Moreover, note that any subset $W \subset \{x',x_1,\ldots,x_n\}$ (containing $x'$ and at least one vertex from $\{x_{k+1},\ldots, x_n\}$ of size $\ell +1$ spans exactly ${\ell \choose 2}+2$ colours.
Note that we may assume $k\leq n/4$, otherwise we are immediately done.
We will have to split our analysis into two main cases.\\

\textbf{Case $1.$ $k\neq n$ or equivalently colour $2$ appears in $H$.}\\

First, we observe that if $\Delta'(v,x')$ has colour $i\notin \{1,2\}$, then condition $(1)$ holds. To see this, denote by $T$ the set $\{x',x_{k+1},\ldots, x_{n}\}$ and let $T'=T\setminus\{x'\}$. Since $T'$ forms a rainbow clique there is at most one vertex, say $x_n\in T$, for which $T\setminus\{x_n\}$ does not span colour $i$. Denote $R=T\setminus\{x_n\}$. There are two possibilities we need to analyze; If $(R\cup\{v\})^{(2)}$ contains $4$ distinct edges $e_1=(a_1,b_1),e_2=(a_2,b_2),e_3=(a_3,b_3)$ and $e_4=(a_4,b_4)$ spanning at most $2$ (without counting colours $\{1,2\}$), then we set $X_0=\{x',v,a_1,b_1,a_2,b_2,a_3,b_3,a_4,b_4\}$ and consider the sets $X_{\ell}=W_{\ell}\cup X_0 $, for every $\ell \in \{1,\ldots,\frac{3n}{4}n-8\}$ where $W_{\ell}\subset R\setminus X_0 $ is any set of size $\ell$. We observe that 
$$ {|X_{\ell}|-2 \choose 2 } +3 \leq \left |\Delta'(X_{\ell})\right| \leq  {\left|X_{\ell}\right|-1 \choose 2} +1.$$
Indeed, $X_{\ell}\setminus\{x',v\}$ spans a rainbow clique not containing colours $\{1,2,i\}$ and $X_{\ell}\setminus\{x'\}$ spans at most ${ |X_{\ell}| \choose 2}-2$ colours. 
This implies that $\mathcal{G}'_{\Delta'}(G)$ has order at least $\frac{7n}{4}-8\geq \frac{5n}{4}$. 
If, on the other hand, there do not exist such $4$ edges, then we can pass to a subset $S \subset R$ of size at least $|R|-4$ for which $S\cup \{v\}\setminus\{x'\}$ forms a rainbow clique not spanning colours $\{1,2,i\}$. In this case, it is easy to check that  $$\left \{{\ell \choose 2}+3: 1 \leq \ell\leq |S|\right \}         \subset \mathcal{G}'_{\Delta'}(G),$$ which implies that $\mathcal{G}'_{\Delta'}(G)$ has order at least $\frac{7n}{4}-4\geq \frac{5n}{4}$, as we wanted to show. We may therefore assume from now on that $\Delta'(x',v)\in \{1,2\}$. 

Since $(3)$ does not hold, there exists a vertex $u\in H$ such that the edge $(v,u)$ has colour $i\notin \{1,2\}$. As argued before, because $H\setminus\{x'\}$ forms a rainbow clique there exists one vertex, say $x_n \in H\setminus\{x',u\}$ such that $H\setminus\{x_n\}$ does not span colour $i$. 
Exactly as in the proof of Lemma~\ref{lem:type 1}, let us assume first the colouring of $H\setminus\{x_n\}\cup \{v\}$ is \textbf{not} a bad colouring of type $(ii)$. This implies that one of the three possibilities must hold:
\begin{enumerate}
    \item[$(a)$] There is an edge $(v,w)$ of colour $1$ or $2$, where $w\in (H\setminus\{x_n,x'\}$);
     \item[$(b)$] There exist a pair of edges $e=(v,w),e'=(a,b)$ of the same colour, where $\{w,a,b\} \subset (H\setminus\{x',x_n\})$;
    \item[$(c)$] The edge $(v,x')$ has colour $j\notin \{1,2\}$.
   
\end{enumerate}

If case $(a)$ holds, set $X_0=\{x',v,u,w\}$ and consider the sets $X_{\ell}=W_{\ell}\cup X_0 $, for every $\ell \in \{1,\ldots,n-4\}$ where $W_{\ell}\subset H\setminus(\{X_0 \cup x_n\}) $ is any set of size $\ell$ containing at least one vertex $z$ for which $(x',z)$ has colour $2$. Observe that 
$$ {|X_{\ell}|-2 \choose 2 } +3 \leq \left |\Delta'(X_{\ell})\right| \leq  {\left|X_{\ell}\right|-1 \choose 2} +1,$$ implying that $\mathcal{G}'_{\Delta'}(G)$ has order at least $2n-4\geq \frac{5n}{4}$. This holds because $X_{\ell}\setminus \{x',v\}$ forms a rainbow clique not spanning colours $\{1,2,i\}$. Moreover, the upper bound follows because $X_{\ell}\setminus \{x'\}$ spans either colour $1$ or colour $2$, by assumption. 
If case $(b)$ holds, set $X_0=\{x',v,u,a,b,w\}$ ($X_0$ could have repeated vertices) and let $X_{\ell}=X_0\cup W_{\ell}$, for every $\ell \in \{1,\ldots, n-5\}$ where $W_{\ell}\subset H\setminus (\{X_0 \cup x_n\})$ and every $W_{\ell}$ contains at least one vertex $z$ for which $(x',z)$ has colour $2$. 
As before, $$ {|X_{\ell}|-2 \choose 2 } +3 \leq \left |\Delta'(X_{\ell}) \right| \leq  {\left|X_{\ell}\right|-1 \choose 2} +1,$$ which implies  $\mathcal{G}'_{\Delta'}(G)$ has order at least $2n-5\geq \frac{5n}{4}$, as we wanted to show. 
Finally, suppose case $(c)$ holds and neither $(a)$ nor $(b)$ holds. As before, note that there is a vertex $x_{n-1}\in H$ for which $H\setminus \{x_{n-1}\}$ does not span colour $j$. In this case, let $X_0=\{x',v\}$ and consider the sets $X_{\ell}=W_{\ell}\cup X_0 $, for every $\ell \in \{1,\ldots,n-3\}$ where $W_{\ell}\subset H\setminus(\{X_0 \cup \{x_n,x_{n-1}\}\}) $ is any set of size $\ell$ containing at least one vertex $z$ for which $(x',z)$ has colour $2$. Observe that 
$$ \left |\Delta'(X_{\ell}) \right| =  {\left|X_{\ell}\right|-1 \choose 2} +3,$$ since $H\setminus\{x',x_n\}$ forms a rainbow clique. Therefore, we have $\mathcal{G}'_{\Delta'}(G)$ has order at least $2n-4\geq \frac{5n}{4}$, as we wanted to show. 

We may now assume the colouring of $H\setminus\{x_n\}\cup \{v\}$ is a bad colouring of type $(ii)$.
We have to split the analysis into two further sub-cases.\\

\textbf{Case $1.1$. $\Delta'((v, H\setminus\{x_n,x'\}))\neq \Delta'((x_n, H\setminus\{x_n,x'\}$))}.\\

Let $y\in H\setminus\{x',x_n\}$ be a vertex such that the edge $(v,y)$ is \textit{uniquely} coloured with colour $j$. Now, since $(2)$ does not hold, there must exist two vertices $a,b\in H\setminus\{x',y\}$ satisfying $\Delta'((v,a))=\Delta'((x_n,b))$. Set $X_0=\{x',v,x_n,a,b,y\}$ and consider the sets $X_{\ell}=W_{\ell}\cup X_0$, for every $\ell \in \{1,2,\ldots, n-5\}$, where $W_{\ell}\subset H\setminus X_0 $ is any set of size $\ell$ containing at least one vertex $z$ for which $(x',z)$ has colour $2$. Note that  $$ {|X_{\ell}|-2 \choose 2 } +3 \leq \left |\Delta'(X_{\ell}) \right| \leq  {\left|X_{\ell}\right|-1 \choose 2} +1.$$ The lower bound follows since $X_{\ell}\setminus\{x',v\}$ is a rainbow clique and does not span colours $\{1,2,j\}$ and the upper bound holds because $X_{\ell}\setminus\{x'\}$ has two edges of the same colour. 
Therefore, $\mathcal{G}'_{\Delta'}(G)$ has order at least $2n-6\geq \frac{5n}{4}$, as we wanted to show. \\

\textbf{Case $1.2$. $\Delta'((v, H\setminus\{x_n,x'\}))= \Delta'((x_n, H\setminus\{x_n,x'\}$))}.\\

By assumption $(4)$ does not hold, hence $v$ is not a \textit{copycat} vertex. Let us assume first there exists $y\in H\setminus\{x',x_n\}$ with $\Delta'((v,y))\neq \Delta'((x_n,y))$ and assume $\Delta'((v,y))=j\notin \{1,2\}$. As seen before, there exists a vertex $x_{n-1}\in H\setminus\{x',x_{n},y\}$ such that $H\setminus\{x_{n-1}\}$ does not span colour $j$. Finally, note that since $\Delta'((v, H\setminus\{x_n,x'\}))= \Delta'((x_n, H\setminus\{x_n,x'\}$ and $n\geq 15$ there exists two vertices $a,b\in H\setminus\{v,x_n,x_{n-1},y\}$ for which $\Delta'((v,a))=\Delta'((x_n,b))$. Set $X_0=\{x',v,x_n,a,b,y\}$ and consider the sets $X_{\ell}=W_{\ell}\cup X_0$, for every $\ell \in \{1,2,\ldots, n-6\}$, where $W_{\ell}\subset H\setminus \{X_0,x_{n-1}\} $ is any set of size $\ell$ containing at least one vertex $z$ for which $(x',z)$ has colour $2$. Observe that $$ {|X_{\ell}|-2 \choose 2 } +3 \leq \left |\Delta'(X_{\ell}) \right| \leq  {\left|X_{\ell}\right|-1 \choose 2} +1.$$ 
Therefore, $\mathcal{G}'_{\Delta'}(G)$ has order at least $2n-6\geq \frac{5n}{4}$, as we wanted to show. 
We may then assume that for every vertex $y\in H\setminus\{x',x_n\}$, $\Delta'((v,y))= \Delta'((x_n,y))$. Therefore, $\Delta'((v,x_n))=j\notin \{1,2\}$. By the same argument as before, there is possibly at most one vertex, say $x_{n-1} \in H\setminus\{x',x_n\}$ such that the unique edge of colour $j$ in $G\setminus\{x_{n-1}\}$ is $(v,x_n)$. Set $X_0=\{x',v,x_n\}$ and consider the sets $X_{\ell}=W_{\ell}\cup X_0$, for every $\ell \in \{1,2,\ldots, n-3\}$, where $W_{\ell}\subset H\setminus \{X_0,x_{n-1}\} $ is any set of size $\ell$ containing at least one vertex $z$ for which $(x',z)$ has colour $2$. Observe that $$ \left |\Delta'(X_{\ell}) \right| ={|X_{\ell}|-2 \choose 2 } +3.$$ 
Therefore, $\mathcal{G}'_{\Delta'}(G)$ has order at least $2n-3\geq \frac{5n}{4}$, as we wanted to show. 
This concludes the analysis of $\textbf{Case 1.}$.\\

\textbf{Case $2.$ Every edge incident with $x'$ in $H$ has colour $1$.}\\

The analysis of this case follows an almost identical argument as the analysis of \textbf{Case $1.$}. For completeness, we opted to include a full proof although we will not add as many details.

First, note that $ \left \{{\ell \choose 2}+1: 1 \leq \ell\leq n+1\right \}\subset \mathcal{G}'_{\Delta'}(G)$. 

Exactly as in \textbf{Case $1$.}, we observe that if $\Delta'(v,x')$ has colour $i\neq 1$, then condition $(1)$ holds. To see this, note that since $H\setminus\{x'\}$ forms a rainbow clique there is at most one vertex, say $x_n\in H\setminus\{x'\}$, for which $H\setminus\{x_n\}$ does not span colour $i$. Denote by $R=H\setminus\{x_n\}$. There are two possibilities we need to analyze; If $(R\cup\{v\})^{(2)}$ contains $4$ distinct edges $e_1=(a_1,b_1),e_2=(a_2,b_2),e_3=(a_3,b_3)$ and $e_4=(a_4,b_4)$ spanning at most $2$ (without counting colour $1$), then we set $X_0=\{x',v,a_1,b_1,a_2,b_2,a_3,b_3,a_4,b_4\}$ and consider the sets $X_{\ell}=W_{\ell}\cup X_0 $, for every $\ell \in \{1,\ldots,n-8\}$ where $W_{\ell}\subset R\setminus X_0 $ is any set of size $\ell$. We observe that 
$$ {|X_{\ell}|-2 \choose 2 } +2 \leq \left |\Delta'(X_{\ell}) \right| \leq  {\left|X_{\ell}\right|-1 \choose 2}.$$
Indeed, $X_{\ell}\setminus\{x',v\}$ spans a rainbow clique not containing colours $\{1,i\}$ and $X_{\ell}\setminus\{x'\}$ spans at most ${ |X_{\ell}| \choose 2}-2$ colours. 
This implies that $\mathcal{G}'_{\Delta'}(G)$ has order at least $2n-9\geq \frac{5n}{4}$. 
If, on the other hand, there do not exist such $4$ edges, then we can pass to a subset $S \subset R$ of size at least $|R|-4$ for which $S\cup \{v\}\setminus\{x'\}$ forms a rainbow clique not spanning colours $\{1,i\}$. In this case, it is easy to check that  $$\left \{{\ell \choose 2}+2: 1 \leq \ell\leq |S|\right \}         \subset \mathcal{G}'_{\Delta'}(G),$$ which implies that $\mathcal{G}'_{\Delta'}(G)$ has order at least $2n-4\geq \frac{5n}{4}$, as we wanted to show.
We may therefore assume from now on that $\Delta'(x',v)=1$. 

Since condition $(3)$ does not hold, there must exist a vertex $u\in H$, for which the edge $(v,u)$ has colour different from $1$, say colour $i$.

As before, we know there exists possibly at most one vertex $x_n\in H\setminus\{x',u\}$ such that $H\setminus \{x_n\}$ does not span colour $i$. Let us assume first the colouring of $H\setminus\{x_n\}\cup \{v\}$ is \textbf{not} a bad colouring of type $(ii)$. This implies that one of the two possibilities must hold:
\begin{enumerate}
    \item[$(b)$] There is an edge $(v,w)$ of colour $1$, where $w\in (H\setminus\{x_n,x'\}$);
    \item[$(c)$] There exist a pair of edges $e=(v,w),e'=(a,b)$ of the same colour, where $\{w,a,b\} \subset (H\setminus\{x',x_n\})$.
\end{enumerate}

If $(a)$ holds, then set $X_0=\{x',u,w,v\}$ and consider the sets $X_{\ell}=W_{\ell}\cup X_0 $, for every $\ell \in \{1,\ldots,n-2\}$ where $W_{\ell}\subset \left(H\setminus\{X_0 \cup x_n\}\right) $ is any set of size $\ell$. Observe that 
$$ {|X_{\ell}|-2 \choose 2 } +2 \leq \left |\Delta'(X_{\ell}) \right| \leq  {\left|X_{\ell}\right|-1 \choose 2},$$ implying that $\mathcal{G}'_{\Delta'}(G)$ has order at least $2n-4\geq \frac{5n}{4}$. 

If $(b)$ holds, set $X_0=\{x',u,v,w\}$ (perhaps $u=x'$) and consider the sets $X_{\ell}=W_{\ell}\cup X_0 $, for every $\ell \in \{1,\ldots,n-2\}$ where $W_{\ell}\subset H\setminus(\{X_0 \cup x_n\}) $ is any set of size $\ell$. As above, observe that 
$$ {|X_{\ell}|-2 \choose 2 } +2 \leq \left |\Delta'(X_{\ell}) \right| \leq  {\left|X_{\ell}\right|-1 \choose 2},$$ implying that $\mathcal{G}'_{\Delta'}(G)$ has order at least $2n-4\geq \frac{5n}{4}$. 

Hence, we may assume the colouring of $H\setminus\{x_n\}\cup \{v\}$ is a \textbf{bad} colouring of type $(ii)$.
Finally, as in \textbf{Case $1.$}, we shall split the analysis into two further sub-cases. \\

\textbf{Case $1.1$. $\Delta'((v, H\setminus\{x_n,x'\}))\neq \Delta'((x_n, H\setminus\{x_n,x'\}$))}.\\

Let $y\in H\setminus\{x',x_n\}$ be a vertex such that the edge $(v,y)$ is \textit{uniquely} coloured with colour $j$. Now, since $(2)$ does not hold, there must exist two vertices $a,b\in H\setminus\{x',y\}$ satisfying $\Delta'((v,a))=\Delta'((x_n,b))$. Set $X_0=\{x',v,x_n,a,b,y\}$ and consider the sets $X_{\ell}=W_{\ell}\cup X_0$, for every $\ell \in \{1,2,\ldots, n-5\}$, where $W_{\ell}\subset H\setminus X_0 $ is any set of size $\ell$. Note that  $$ {|X_{\ell}|-2 \choose 2 } +2 \leq \left |\Delta'(X_{\ell}) \right| \leq  {\left|X_{\ell}\right|-1 \choose 2}.$$ The lower bound follows since $X_{\ell}\setminus\{x',v\}$ is a rainbow clique and does not span colours $\{1,j\}$ and the upper bound holds because $X_{\ell}\setminus\{x'\}$ has two edges of the same colour. 
Therefore, $\mathcal{G}'_{\Delta'}(G)$ has order at least $2n-6\geq \frac{5n}{4}$, as we wanted to show. \\

\textbf{Case $1.2$. $\Delta'((v, H\setminus\{x_n,x'\}))= \Delta'((x_n, H\setminus\{x_n,x'\}$))}.\\

By assumption $(4)$ does not hold, hence $v$ is not a \textit{copycat} vertex. Let us assume first there exists $y\in H\setminus\{x',x_n\}$ with $\Delta'((v,y))\neq \Delta'((x_n,y))$. Let $\Delta'((v,y))=j$ and by assumption $j\neq 1$. As seen before, there exists a possibly at most one vertex, say $x_{n-1}\in H\setminus\{x',x_{n},y\}$ such that $H\setminus\{x_{n-1}\}$ does not span colour $j$. Finally, note that since $\Delta'((v, H\setminus\{x_n,x'\}))= \Delta'((x_n, H\setminus\{x_n,x'\}$ and $n\geq 15$ there exists two vertices $a,b\in H\setminus\{v,x_n,x_{n-1},y\}$ for which $\Delta'((v,a))=\Delta'((x_n,b))$. Set $X_0=\{x',v,x_n,a,b,y\}$ and consider the sets $X_{\ell}=W_{\ell}\cup X_0$, for every $\ell \in \{1,2,\ldots, n-6\}$, where $W_{\ell}\subset H\setminus \{X_0,x_{n-1}\} $ is any set of size $\ell$. Observe that $$ {|X_{\ell}|-2 \choose 2 } +2 \leq \left |\Delta'(X_{\ell}) \right| \leq  {\left|X_{\ell}\right|-1 \choose 2}.$$ 
Therefore, $\mathcal{G}'_{\Delta'}(G)$ has order at least $2n-4\geq \frac{5n}{4}$, as we wanted to show. 
We may then assume that for every vertex $y\in H\setminus\{x',x_n\}$, $\Delta'((v,y))= \Delta'((x_n,y))$. Therefore, as $(3)$ does not hold, $\Delta'((v,x_n))=j$, for some $j\neq 1$. By the same argument as before, there is possibly at most one vertex, say $x_{n-1} \in H\setminus\{x',x_n\}$ such that the unique edge of colour $j$ in $G\setminus\{x_{n-1}\}$ is $(v,x_n)$. Set $X_0=\{x',v,x_n\}$ and consider the sets $X_{\ell}=W_{\ell}\cup X_0$, for every $\ell \in \{1,2,\ldots, n-2\}$, where $W_{\ell}\subset H\setminus \{X_0,x_{n-1}\} $ is any set of size $\ell$. Observe that $$ \left |\Delta'(X_{\ell}) \right|= {|X_{\ell}|-2 \choose 2 } +2.$$ 
Therefore, $\mathcal{G}'_{\Delta'}(G)$ has order at least $2n-2\geq \frac{5n}{4}$, as we wanted to show. 
This concludes the analysis of $\textbf{Case 2.}$ and the proof of Lemma~\ref{lem:type 1}.
\end{proof}

Armed with the previous lemmas we are now ready to prove our main theorem. 

\begin{proof}[\textbf{Proof of Theorem~\ref{thm:main2}}.]
Let $\alpha'=\min\{\alpha, 1/100\}$, where $\alpha$ is given by Proposition~\ref{prop:basecase}. The proof will go by induction on $m$, the number of colours used by $\Delta'$. 
The base cases follow from Proposition~\ref{prop:basecase}.
We shall assume that the theorem holds for every integer $3\leq m'<m$ and we want to show it also holds for $m\geq 200$. 
Let $G=K_{n+1}$ be the complete graph with a special $m$-colouring $\Delta'$ and with the special vertex $x'$, and $\Delta'$ is not a bad colouring. Our aim is to prove that
$\left |\mathcal{G}'_{\Delta'}(G)\right|\geq (1+\alpha')\sqrt{2m}$. 

\begin{claim}\label{claim:degreesmall}
For every $v\in V(G')$, $\rho(v)\leq \left \lfloor(1+\alpha')\sqrt{2m}\right \rfloor$.
\end{claim}

\begin{proof}
Suppose there is some $v$ satisfying $\rho(v)> \left \lfloor(1+\alpha')\sqrt{2m} \right \rfloor$. Let $w_1,w_2,\ldots w_{t}$ (where $t=\left \lceil(1+\alpha')\sqrt{2m} \right \rceil$) be some neighbours of $v$ for which the colours of the edges $(v,w_i)$ have distinct colours for distinct $i\in [t]$ and these colours only appear on edges incident with $v$.
We define the sets $W_i=\left \{v\right \}\cup \left \{\bigcup_{j=1}^{i} w_j \right\} \cup \{x'\}$, for every $i \in \{1,2,\ldots, t\}$. It follows from the definition of $\rho$ that $\gamma_{\Delta'}\left (W_j\right ) > \gamma_{\Delta'}\left (W_i\right)$, for every $j\geq i$, implying $\left |\mathcal{G}'_{\Delta'}(G)\right|\geq (1+\alpha')\sqrt{2m}$.
\end{proof}

\begin{claim}\label{claim:samedegree}
There exists a positive integer $\rho$ such that $\rho(v)=\rho$, for every $v \in G'$. 
\end{claim} 

\begin{proof}
Suppose there exist two distinct vertices $v,w\in G'$ with $\rho(v)>\rho(w)$. We shall assume $\rho(v)=\max\left \{\rho(v):v\in G'\right\}=\rho$. Consider the induced coloured complete subgraph $W=G \setminus \{v\}$. There might exist some vertices $z\in W\setminus\{x'\}$ such that $\rho_{W}(z)=0$. In this case, we keep deleting such vertices one by one until the remaining induced subgraph $W_{0}$ satisfies $W_{0}^{\geq 1}=W_{0}$. This procedure never deletes all edges of some colour present on $W$, hence $\gamma=\gamma_{\Delta'}(W)=\gamma_{\Delta'}(W_{0})$. Let us denote by $B$ the set of deleted vertices, namely $W\setminus W_{0}$. 

Since by assumption $\rho_{W_0}(v)\geq 1$ and $\gamma_{\Delta'}\left (W_{0}\right )<m$, we may apply the induction hypothesis to $W_{0}$.
First, we shall assume the colouring on $W_{0}$ is not a bad colouring. Then, we have that the total number of colours appearing in $W_{0}$ is at least $m-(1+\alpha')\sqrt{2m}$, as by Claim~\ref{claim:degreesmall}, $\rho\leq \left \lfloor(1+\alpha')\sqrt{2m}\right \rfloor$. Therefore, $\left |\mathcal{G}'_{\Delta'}\left (W_{0}\right)\right |\geq (1~+~\alpha')~\sqrt{{2m-2(1+\alpha')\sqrt{2m}}}$. Observe, that $\gamma$ is the largest value in $ \mathcal{G}'_{\Delta'}(W_{0})$. 
Consider now the subgraph $\left (W\setminus\{w\}\right) \cup\{v\}$ and notice that $\gamma_1=\gamma_{\Delta'}\left (\left (W\setminus\{w\}\right)\cup\{v\}\right)$ is greater than $\gamma$, because $\rho(v)>\rho(w)$.

Finally, by considering the entire graph $G$, we obtain yet another value $\gamma_{\Delta'}(G)=\gamma_2\notin\mathcal{G}'_{\Delta'}\left (\left (W\setminus \{w\}\right) \cup\{v\}\right)$, because $\rho(w)\geq 1$. We have just proved that $\gamma_2>\gamma_1>\gamma$. Therefore, \[
 |\mathcal{G}'_{\Delta'}(G)|\geq 2+ (1+\alpha')\sqrt{{2m-2(1+\alpha')\sqrt{2m}}}\geq (1+\alpha')\sqrt{{2m}}.\] 
 The second inequality holds, by a simple calculation, for every $m\geq 5$, as long as $\alpha'\leq \frac{1}{10}$.

Suppose the colouring induced by $\Delta'$ on $W_{0}$ is a bad colouring. We split our analysis into two cases, depending on the type of the bad colouring.\\

\textbf{Case $1$. The colouring induced on $W_{0}$ is a bad colouring of type $(i)$.}\\

 First, let us suppose that $W_{0}=W$. Since $W_{0}$ is a rainbow complete graph we have that for every vertex $v \in W_0$, $\rho_{W}(v)\geq n-1$ and in particular $\rho(w)\geq n-1$. As $\rho=\rho(v)>\rho(w)$ then $\rho$ must be equal to $n$. However, in this case the colouring $\Delta'$ on $G$ would also be of type $(i)$, which is a contradiction. 

We may then assume $W_0\neq W$. Recall $B=W\setminus W_{0}=\{z_1,\ldots,z_b\}$. Observe that for every vertex $z\in B$, $\rho_{G}(z)=1$ and the edge $(v,z)$ is \textit{uniquely} coloured. Denote by $n_0$ the size of $W_0$ and note that $n_0=n-b$.

Now, since $m\leq {n_0 \choose 2}+ \rho(v) +1$, we must have
\begin{equation}
n_0 \geq \sqrt{2m-2(1+\alpha')\sqrt{2m}}. 
\end{equation}
Since $m\geq 100$ it follows that $n_0\geq \frac{9\sqrt{2m}}{10}$. 
Let us write $V(W_{0})=\{x_1=x',x_2,\dots, x_{n_0}\}$. Note that $\mathcal{G}'_{\Delta'}(G)$ contains every integer of the form $ { \ell \choose 2}+1$, for every $\ell \in \left \{1,2,\ldots,n_0 \right\}$, because $W_0$ is a rainbow complete graph. 

We shall now apply Lemma~\ref{lem:type 1} to $W_0 \cup \{v\}$. Note that if $(1)$ in~Lemma~\ref{lem:type 1} holds then we are done since in that case we would have $\left |\mathcal{G}'_{\Delta'}(G)\right|\geq \frac{5n_0}{4}\geq (1+\alpha')n$. Therefore, we may assume that either $(2)$ or $(3)$ hold. \\
Suppose first that $(2)$ holds. Then, we shall apply again Lemma~\ref{lem:type 1} to $(W_0\cup\{v\})\cup\{z_1\}$. As before if condition $(1)$ holds we are done. Moreover, $(3)$ can not hold because the edge $(z_1,v)$ is uniquely coloured. Hence, $(2)$ must hold again. By continuing this process for the rest of vertices of $B$ we obtain that either $\left |\mathcal{G}'_{\Delta'}(G)\right|\geq \frac{5n_0}{4}\geq (1+\alpha')\sqrt{2m}$ or the entire $G$ is a bad colouring of type $(i)$, which is a contradiction. \\
Finally, let us suppose $(3)$ holds when we apply Lemma~\ref{lem:type 1} to $W_0\cup \{v\}$. 
From Lemma~\ref{lem:type 1} applied to the graph $W_0\cup \{z_1\}$, we may assume, as before that $(2)$ or $(3)$ hold. If $(3)$ holds then $\mathcal{G}'_{\Delta'}(G)$ contains every integer in $\left \{ { \ell \choose 2}+2: 1 \leq \ell\leq n_0\right\} $ because any subgraph $S\subset W_0\cup \{v\}$ of size $\ell+2$ and containing $\{x',z_1,v\}$ spans ${ \ell \choose 2}+2$ colours. Hence, $\left |\mathcal{G}'_{\Delta'}(G)\right|\geq 2n_0\geq (1+\alpha')\sqrt{2m}$. Similarly if $(3)$ holds then $\mathcal{G}'_{\Delta'}(G)$ contains as well every integer in $\left \{ { \ell \choose 2}+2: 1 \leq \ell\leq n_0\right \} $ because any subset $S\subset W_0\cup\{v\}$ containing $\{x',z_1,v\}$ of size $\ell+1$ spans ${\ell \choose 2}+2$ colours. This concludes the analysis of Case $1$. \\

\textbf{Case $2$. The colouring induced on $W_{0}$ is a bad colouring of type $(ii)$.}\\

The analysis of this case follows an almost identical strategy as the one we took in Case $1$. 

First, let us suppose $W_0=W$. Recall that by definition the only colours present on the edges of $W$ incident with the special vertex $x'$ are colours $1$ or $2$ and colour $2$ does not appear in $\Delta'(W')$. As $\rho=\rho(v)=\max\{\rho(z): z\in G\}$, it is easy to see that $\rho$ must be at least $n-1> 3$. We shall now apply Lemma~\ref{lem:type 2} to $W \cup \{v\}$. Note that if $(1)$ holds then we are done. Moreover, by assumption on $\rho$, $(3)$ can not hold. Hence, $(2)$ must hold and the colouring $\Delta'$ would have to be a bad colouring of type $(ii)$ on the entire $G$, which is a contradiction.

We may then assume $W_0$ is strictly contained in $W$. 
Denote by $n_0$ the size of $W_0$ and as before, let $B=W\setminus W_0=\{z_1,z_2,\ldots,z_b\}\neq \emptyset$, where $n_0=n-b$.

Since $m\leq {n_0 \choose 2}+ \rho(v) +2$, we must have
\begin{equation}
n_0 \geq \sqrt{2m-2(1+\alpha')\sqrt{2m}-2}. 
\end{equation}
Since $m\geq 100$ it follows that $n_0\geq \frac{9\sqrt{2m}}{10}$.
Let $V(W_0)=\{x_1=x',x_2,\dots, x_{n_0}\}$ and denote by $k\in [n_0]$ the positive integer such that every edge $(x',x_i)$ has colour $1$ if $i\leq k$, and $(x',x_i)$ has colour $2$ if $i>k$. 
Observe that
\begin{equation}\label{eq:obvious2}
    \left \{{\ell \choose 2}+1: 1 \leq \ell\leq k\right \} \cup  \left \{{\ell \choose 2}+2: 2 \leq \ell\leq n_0+1\right \}         \subset \mathcal{G}'_{\Delta'}(G).
\end{equation}
Indeed, any subset $W \subset \{x',x_1,\ldots,x_k\}$ of size $\ell+1$ (containing $x'$) spans exactly ${\ell \choose 2}+1$ colours. Furthermore, any subset $W \subset \{x',x_2,\ldots,x_{n_0}\}$ of size $\ell +1$ (containing $x'$ and at least one vertex from $\{x_{k+1},\ldots, x_{n_0}\}$) spans exactly ${\ell \choose 2}+2$ colours. This shows that we may assume either $k=n_0$ or $k \leq \frac{n_0}{8}$ otherwise $\left |  \mathcal{G}'_{\Delta'}(G)\right | \geq \frac{9n_0}{8}\geq (1+\alpha')\sqrt{2m}$. 

Now, we shall apply Lemma~\ref{lem:type 2} to $W_0 \cup \{v\}$. Note that if $(1)$ in~Lemma~\ref{lem:type 2} holds then we are done since in that case we would have $\left |\mathcal{G}'_{\Delta'}(G)\right|\geq \frac{5n_0}{4}\geq (1+\alpha')\sqrt{2m}$. Therefore, we may assume that either $(2)$ or $(3)$ must hold.

Suppose first that $(2)$ holds. Then, we shall apply again Lemma~\ref{lem:type 2} to $(W_0\cup\{v\})\cup\{z_1\}$. As argued before, if condition $(1)$ holds we are done. Moreover, $(3)$ can not hold because the edge $(z_1,v)$ is \textit{uniquely} coloured. Hence, $(2)$ must hold again. By continuing this way for the remaining vertices of $B$, we obtain that either $\left |\mathcal{G}'_{\Delta'}(G)\right|\geq \frac{5n_0}{4}\geq (1+\alpha')\sqrt{2m}$ or the entire $G$ is a bad colouring of type $(ii)$, which is a contradiction. 

We shall then assume $(3)$ holds when we apply Lemma~\ref{lem:type 2} to $W_0\cup \{v\}$. 
From Lemma~\ref{lem:type 2} applied to the graph $W_0\cup \{z_1\}$, we may assume assume that $(2)$ or $(3)$ hold, by the exact same argument as before. If $(3)$ holds then $\mathcal{G}'_{\Delta'}(G)$ contains every integer in $\left \{ { \ell \choose 2}+2: 1 \leq \ell\leq n_0\right\} $ because any subgraph $S\subset W_0\cup \{v\}$ of size $\ell+2$ and containing $\{x',z_1,v\}$ spans ${ \ell \choose 2}+2$ colours. Hence, $\left |\mathcal{G}'_{\Delta'}(G)\right|\geq 2n_0\geq (1+\alpha')n$. Similarly if $(3)$ holds then $\mathcal{G}'_{\Delta'}(G)$ contains as well every integer in $\left \{ { \ell \choose 2}+2: 1 \leq \ell\leq n_0\right \} $ because any subset $S\subset W_0\cup\{v\}$ containing $\{x',z_1,v\}$ of size $\ell+1$ spans ${\ell \choose 2}+2$ colours. This concludes the analysis of Case $2$.

 would imply $|\mathcal{G}'_{\Delta'}(G)|\geq (1+\alpha')\sqrt{2m}$, obtaining a contradiction. \\

This concludes the analysis of \textbf{Case} $2.$ and the proof of Claim~\ref{claim:samedegree}. 
\end{proof}

From now on, we shall assume $\rho(v)=\rho$, for every $v \in V(G')$. We claim that $\rho$ can not be \textit{too} small. 

\begin{claim}\label{claim:rhobig}
Let $\beta=(2+3\alpha')-\frac{2}{1+3\alpha'}$. Then, $\rho \geq (1-\beta)\sqrt{2m}$.
\end{claim}
\begin{proof}
Suppose $\rho < (1-\beta)\sqrt{2m}$. Let $A_1, A_2,\ldots, A_t$ be the $t$ sets obtained by applying Lemma~\ref{lemma: finitecomplete} to $G$. Observe that $\bigcup_{i=1}^{t} V(A_i)=V(G)$, because by property $(3)$ in Lemma~\ref{lemma: finitecomplete} all colours appear in $G[\cup_{i=1}^{t} A_i]$, but $\rho(v) \geq 1$, for every $v\in G\setminus\{x'\}$.

Our aim is to show that $t$ must be at least $(1+\alpha')\sqrt{2m}$, thus proving what we wanted to show.
To do so, we shall need to give a lower bound for $n+1$, the total number of vertices of $G$. 
Let $m_{0}\leq m $ be the number of colours which appear in exactly one edge of $G$ in other words, the number of \textit{uniquely} coloured edges. Clearly, $\frac{n\rho}{2}\geq m_{0}$. Moreover, each of the remaining colours must appear in at least $2$ edges of $G$. Hence, the following holds.
\begin{align*}
 m_{0}+ 2(m-m_{0})\leq {n+1 \choose 2} & \implies 
 2m-\frac{n\rho}{2} \leq {n+1\choose 2} 
 \\& \implies 4m \leq \left( n+1 + (1-\beta)\sqrt{2m} \right) n.
\end{align*}
Since  $\beta = (2+3\alpha')-\frac{2}{1+3\alpha'}$, we have that $\left( (1-\beta) +(1+3\alpha') \right) (1+ 3\alpha') \leq 2$, implying $n$ must be at least $(1+3\alpha')\sqrt{2m}$. 

For every $j\in [t]$, recall $c(j)$ denotes the order of the set $\Delta' \left(\bigcup_{i=1}^{j} A_i \right)\setminus \Delta'\left(\bigcup_{i=1}^{j-1} A_i\right)$. 
Let $T_1\subseteq [t]$ be the set of indices for which $A_j$ has size $1$ and let $T_2=[t]\setminus T_1$ be the set of indices for which $A_j$ has size $2$. Denote by $t_1$ and $t_2$, the sizes of $T_1$ and $T_2$, respectively.  
First, we shall assume that $t_2\leq 2\alpha' \sqrt{2m}$. 
In this case we have, 
$$n+1\leq t_1+2\cdot t_2=t+t_2\implies t\geq (1+\alpha')\sqrt{2m},$$ which is what we wanted to show.

We may then assume that $t_2\geq 2\alpha'\sqrt{2m}$. Suppose, for contradiction, $t\leq (1+\alpha')\sqrt{2m}$. Let $\tau: [t]\rightarrow [t]$ be a function defined as $\tau(i)=|\{ j\leq i: j \in T_2\}|$. By $(4)$ in Lemma~\ref{lemma: finitecomplete}, we have that $c(i)=1$ whenever $ |A_i|=2$. Moreover, it is easy to deduce from $(4)$ and $(5)$ in Lemma~\ref{lemma: finitecomplete}, that $c(i)\leq i-\tau(i)$, for every $j\in [t]$. 

The following holds: 
 \begin{align*}
m \leq  \sum_{i=1}^{t} 
c(i)&\leq \sum_{i\in T_1} (i-\tau(i)) +t_2 \leq {t_1 \choose 2} + t_2 \leq \frac{t_1^{2}}{2}+t_2\\
&\implies 2m \leq t_1^{2}+2t_2=(t-t_2)^{2}+2t_2 \implies\\ 
& 0 \leq ((1+\alpha')\sqrt{2m}-t_2)^{2}+2t_2-2m.
\end{align*}
Now, for every $m\geq 10$, the right hand side is a quadratic function on $t_2$ which is negative for every $ 2\alpha'\sqrt{2m} \leq t_2 \leq t $ and this contradicts the fact $t\leq (1+\alpha')\sqrt{2m}$.

Hence, $t\geq (1+\alpha')\sqrt{2m}$ and we obtain $|\mathcal{G}'_{\Delta'}(G)|\geq (1+\alpha')\sqrt{2m}$. So we may assume $\rho\geq (1-\beta)\sqrt{2m}$ and we finish the proof of Claim~\ref{claim:rhobig}. 
\end{proof}
 
Claim~\ref{claim:rhobig} easily implies the number of vertices in $G$ can not be \textit{too} large. 

\begin{claim}\label{claim:numbervert}
$|V(G')|=n\leq (1-\beta)^{-1}\sqrt{2m}$. 
\end{claim}
\begin{proof}
Clearly, there are at least $\frac{n\cdot\rho}{2}$ distinct colours in $G$, hence $m\geq \frac{n\cdot\rho}{2}$. Using the bound on $\rho$ in Claim~\ref{claim:rhobig}, we have $n \leq (1-\beta)^{-1}\sqrt{2m}$. 
\end{proof}

\begin{claim}\label{claim:triangle}
There exist three vertices $w,y,z\in G'$ such that both edges $(w,y)$,$(w,z)$ are \textit{uniquely} coloured but the edge $(y,z)$ is not.
\end{claim}
\begin{proof}
We shall prove this claim by considering the subgraph $G^{*}\subset G'$ consisting of the \textit{uniquely} coloured edges. Indeed, suppose that $G^{*}=G'$. Since $G'$ satisfies the conditions in Lemma~\ref{lem:uniquelycolour} we obtain $|\mathcal{G}'_{\Delta'}(G)|\geq \frac{5}{4}\sqrt{2m} \geq (1+\alpha')\sqrt{2m}$, which is what we wanted to show. We may then assume $G^{*}\neq G'$.

We claim that the degree of any vertex in $G^{*}$ is at least $\frac{n}{2}$. To see this, let $un(v)$ be the number of uniquely coloured edges incident with a vertex $v\in G'$. Since $\rho\geq (1-\beta)\sqrt{2m}$, for any vertex $v\in G'$, it must be incident with at least 
$2((1-\beta)\sqrt{2m}-un(v))+un(v)=2(1-\beta)\sqrt{2m}-un(v)$ vertices.

Suppose $un(v) \leq n/2$, for some vertex $v\in G'$. Then, $un(v) \leq \frac{1}{2}(1-\beta)^{-1}\sqrt{2m}$, implying $n\geq 2(1-\beta)\sqrt{2m}-\frac{1}{2}(1-\beta)^{-1}\sqrt{2m}> (1-\beta)^{-1}\sqrt{2m}$, which contradicts Claim~\ref{claim:numbervert}. (The last inequality holds if $\beta \leq \frac{1}{10}$). 
Note that this implies there must exist an induced path of length $2$, $ywz$, in $G^{*}$ because the graph $G^{*}$ is not the union of disjoint complete subgraphs. We have thus found the desired three vertices, namely $y,w,z$. This concludes the proof of Claim~\ref{claim:triangle}. 
\end{proof}

Let $w,y,z$ be such a triple of vertices as in Claim~\ref{claim:triangle}.
Consider the complete subgraph $W=V(G)\setminus \{z,y\}\subset G$. 
Note that by Claim~\ref{claim:rhobig}, we must have $\rho>2$, because $m\geq 10$. Therefore, $W^{\geq 1}=W$.

In the same way as in the proof of Claim~\ref{claim:samedegree}, we shall have to consider two cases. First, we consider the case when the colouring induced by $\Delta'$ on $W$ is not a \textit{bad} colouring.\\

\textbf{Case $1.$ The colouring induced on $W$ is not a bad colouring.}\\

Our aim is to show the existence of three distinct values $\gamma_1,\gamma_2,\gamma_3 \in [m]$ which belong to $\mathcal{G}'_{\Delta'}(G)$ but are not in $\mathcal{G}'_{\Delta'}(W)$. Denote by $\gamma=\gamma_{\Delta'}(W)$. Now, we will look at the following three complete subgraphs of $G$. 
\begin{itemize}
\item $W_1=(W\setminus \{w\}) \cup \{z\}$;
\item $W_2=(W \setminus \{w\})\cup \{z,y\}$;
\item $W_3=G$.
\end{itemize}

Observe that $\gamma(W_1)=\gamma +1$. Indeed, removing $w$, deletes exactly $\rho-2$ colours from $\Delta'(W)$ (namely those colours which are only incident with $w$) but adding $z$ to $W\setminus\{w\}$ increases the total number of colours by $\rho-1$ as the edge $(y,z)$ is not \textit{uniquely} coloured. Furthermore, by the same reasoning, $\gamma(W_2)=\gamma+\rho$. Finally, it is easy to see $\gamma(W_3)=\gamma(G)>\gamma +\rho$. Therefore, these three values are distinct and are strictly bigger than $\gamma$. Note that the number of colours spanned by $W$ is at least $m-2\rho\geq m-2(1+\alpha')\sqrt{2m}$. By the induction hypothesis on $W$ we have,

\[ |\mathcal{G}'_{\Delta'}(G)|\geq (1+\alpha')\sqrt{2m-4(1+\alpha')\sqrt{2m}}+3>(1+\alpha')\sqrt{2m},
\]as we wanted to show. 
The last inequality holds for every $m\geq 15$ and $\alpha'\leq \frac{1}{20}$. \\

\textbf{ Case $2.$ The colouring induced on $W$ is a bad colouring.}\\

Let the bad colouring induced on $W$ be of type $(i)$, hence $W$ forms a rainbow complete graph of size $n-1$, where $n\geq 10$.
We claim $\rho$ must be equal to $n$ which implies the colouring of $G$ would also be a bad colouring of type $(i)$, which is a contradiction. Indeed, if $\rho<n$ then, for every vertex $v\in G'$, $v$ is incident with an edge whose colour $c(v)$ appears in at least two distinct edges of $G$ or $v$ is incident with an edge of colour $1$. However, because $W$ forms a rainbow graph, for very vertex $v\in W'$ at least one of the edges of colour $c(v)$ or the edge of colour $1$ incident with $v$ must be also incident with $y$ or $z$, and this is clearly impossible to hold for every vertex in $W$, since $\rho(y)=\rho(z)=\rho$.  

We may then assume the colouring on $W$ is a bad colouring of type $(ii)$.
 Note that the total number of colours appearing in $W$ is at least $m-2\rho\geq m-4(1+\alpha')\sqrt{2m}$. Now, we shall apply Lemma~\ref{lem:type 2} to $W\cup\{z\}$. Clearly by Claim~\ref{claim:rhobig}, $\rho(z)=\rho>3$ and $(3)$ can not hold. Moreover, if $(1)$ holds then we are done because $\frac{5}{4}(n-2)\geq (1+\alpha')\sqrt{2m}$. Hence, $(2)$ must hold and we obtain that $W\cup\{z\}$ must induce a bad colouring of type $(ii)$. Finally, we apply once again Lemma~\ref{lem:type 2} to $(W\cup\{z\})\cup \{y\}$. As before, if $(1)$ holds we are done and $(3)$ can not hold. So we obtain $G$ forms a bad colouring of type $(ii)$, which is a contradiction.
To conclude the proof, we note that if $\alpha'\leq \frac{1}{100}$  then $\beta\leq \frac{1}{10}$ and all the inequalities in the proof hold.
This completes the proof of Theorem~\ref{thm:main2}. 
\end{proof}

\section{Concluding remarks}
The first remark we would like to mention concerns the value of $\alpha'$ in the main theorem. Although we are confident our methods could be sharpened to get a larger constant $\alpha'$, we doubt our approach could be improved to get the correct value, even when $m={ n\choose 2}+3$, for some integer $n\geq 2$. 

Secondly, we remark that our theorem is sharp (up to the value of $\alpha'$) for values of $m$ \textit{near} the boundary of intervals of the form $\left [ {n \choose 2}+3, {n+1 \choose 2}\right]$ (for some $n\geq 3$). To see this, one may take a bad colouring of type $(ii)$ and recolour a constant number of edges incident with the special vertex $x'$ with new and distinct colours. It is easy to check this procedure increases the size of $\mathcal{G}$ by at most some multiplicative factor. 
However, whenever $m$ is \textit{far} from the boundaries of such intervals we conjecture $\psi$ actually changes its behaviour. 

\begin{conjecture}
For every $C>0$, there exist a positive integer $\ell$ such that for every $m\in \mathbb{N}$, where $ { n \choose 2}+\ell \leq m \leq  { n+1 \choose 2}-\ell$, $$ \psi(m) \geq C\sqrt{2m}.$$ 
\end{conjecture}

Recall that it is still unknown whether $\psi(m)=o(m)$ as $m\rightarrow \infty$. We remind the reader of Narayanan's conjecture. 

\begin{conjecture}[Narayanan]
$\psi(m)=o(m)$ as $m\rightarrow \infty$. 
\end{conjecture}

It seems plausible that the function $\psi$ is unimodal within those intervals. 
Finally, we remark that it is possible to deduce from our theorems a characterization of the $m$-colourings $\Delta$ of $\mathbb{N}^{(2)}$ for which $|\mathcal{F}_{\Delta}|=\psi(m)$, for every $m\in \left\{{n \choose 2}+1,{n \choose 2}+2: \text{ for } n\text{ sufficiently large}\right\}$. We are able to prove $\Delta$ must be (up to a permutation of the naturals) one of the two colourings described in the Introduction.  

\section{Acknowledgments}
I would like to thank B\'ela Bollob\'as, Teeradej Kittipassorn and Bhargav Narayanan and the anonymous reviewers for their helpful comments. 
Part of this project was carried through during my stay at IMT School for Advanced Studies Lucca. I would like to thank Prof. Caldarelli for his kind hospitality.  

\bibliographystyle{amsplain}
\bibliography{manysizecolours.bib}
\end{document}